\documentclass[12pt]{amsart}%
\usepackage{amsmath}%
\usepackage{amssymb}%
\usepackage{amscd}%
\usepackage{color}%
\usepackage{enumitem}%
\usepackage{pifont}%
\usepackage{ifsym}%
\usepackage{graphicx}%
\usepackage{mathrsfs}%
\usepackage[all]{xy}%
\usepackage{stmaryrd}
\usepackage{newtxtext,newtxmath}%
\usepackage[all, warning]{onlyamsmath}
\usepackage{amsrefs}
\usepackage{bbm}
\usepackage[capitalise]{cleveref}
\usepackage{stmaryrd}
\usepackage{trimclip}
\makeatletter
\DeclareRobustCommand{\arr}{%
  \mathrel{\mathpalette\short@to\relax}%
}
\newcommand{\short@to}[2]{%
  \mkern2mu
  \clipbox{{.3\width} 0 0 0}{$\m@th#1\vphantom{+}{\shortrightarrow}$}%
  }
\makeatother
\theoremstyle{plain}
    
    \newtheorem{theorem}{Theorem}[section]
    \newtheorem{proposition}[theorem]{Proposition}
    \newtheorem{lemma}[theorem]{Lemma}

    \newtheorem*{question*}{Question}
\theoremstyle{definition}
    \newtheorem{definition}[theorem]{Definition}
    
    \newtheorem{example}[theorem]{Example}
    \newtheorem{remark}[theorem]{Remark}
      
\def\Alphabet{A,B,C,D,E,F,G,H,I,J,K,L,M,N,O,P,Q,R,S,T,U,V,W,X,Y,Z}
\def\alphabet{a,b,c,d,e,f,g,h,i,j,k,l,m,n,o,p,q,r,s,t,u,v,w,x,y,z}
\def\endpiece{xxx}
\def\makeAlphabet[#1]{\expandafter\makeA#1,xxx,}
\def\makealphabet[#1]{\expandafter\makea#1,xxx,}
\def\makeA#1,{\def\temp{#1}\ifx\temp\endpiece\else%
\mkbb{#1}\mkfrak{#1}\mkbf{#1}\mkcal{#1}\mkscr{#1}\mkbs{#1}\expandafter\makeA\fi}%
\def\makea#1,{\def\temp{#1}\ifx\temp\endpiece\else\mkfrak{#1}\mkbf{#1}\mkbs{#1}\expandafter\makea\fi}%
\def\mkbb#1{\expandafter\def\csname bb#1\endcsname{\mathbb{#1}}}
\def\mkfrak#1{\expandafter\def\csname fr#1\endcsname{\mathfrak{#1}}}
\def\mkbf#1{\expandafter\def\csname b#1\endcsname{\mathbf{#1}}}
\def\mkcal#1{\expandafter\def\csname c#1\endcsname{\mathcal{#1}}}
\def\mkscr#1{\expandafter\def\csname s#1\endcsname{\mathscr{#1}}}
\def\mkbs#1{\expandafter\def\csname bs#1\endcsname{{\boldsymbol{#1}}}}
\def\makeop[#1]{\xmakeop#1,xxx,}
\def\mkop#1{\expandafter\def\csname #1\endcsname{{\mathrm{#1}}}} %
\def\xmakeop#1,{\def\temp{#1}\ifx\temp\endpiece\else\mkop{#1}\expandafter\xmakeop\fi}%
\def\makeup[#1]{\xmakeup#1,xxx,}
\def\mkup#1{\expandafter\def\csname #1\endcsname{{\mathrm{#1}\,}}} %
\def\xmakeup#1,{\def\temp{#1}\ifx\temp\endpiece\else\mkup{#1}\expandafter\xmakeup\fi}%
\makeAlphabet[\Alphabet]
\makealphabet[\alphabet]
\makeop[Map,Hom,alt,id,loc,col,unif,len,Supp,Consv,pr,Aut,Ext,ab,tors,Coind,univ,pro,av,sp,rad,ex,Fr,Span,tor,diam]
\makeup[Im,Ker,Coker,rank]

\def\e{\eta}

\def\La{\Lambda}
\def\nabe{\nabla_{\!e}}
\def\FC{\sF\kern-0.5mm\sC}

\def\R{\bbR}
\def\Z{\bbZ}
\def\C#1{\Consv^\phi(#1)}

\def\abs#1{|#1|}

\def\Dabs#1{\bigl|\kern-0.3mm\bigl|#1\bigr|\kern-0.3mm\bigr|}

\def\vp{\varphi}

\def\N{\bbN}
\def\bsxi{\boldsymbol{\xi}}

\begin{document}

\setcounter{tocdepth}{1}

\title[Crystal Lattices]{Varadhan's Decomposition of Shift-Invariant Closed Uniform Forms for Large Scale Interacting Systems on General Crystal Lattices}
\author[Bannai]{Kenichi Bannai$^{*\diamond}$}\email{bannai@math.keio.ac.jp}
\author[Sasada]{Makiko Sasada$^{\star\diamond}$}\email{sasada@ms.u-tokyo.ac.jp}
\thanks{This work was supported in part by JST CREST Grant Number JPMJCR1913, KAKENHI 18H05233,
and the UTokyo Global Activity Support Program for Young Researchers.}
\address{${}^*$Department of Mathematics, Faculty of Science and Technology, Keio University, 3-14-1 Hiyoshi, Kouhoku-ku, Yokohama 223-8522, Japan.}

\address{${}^\star$Department of Mathematics, University of Tokyo, 3-8-1 Komaba, Meguro-ku, Tokyo 153-8914, Japan.}
\address{${}^\diamond$Mathematical Science Team, RIKEN Center for Advanced Intelligence Project (AIP),1-4-1 Nihonbashi, Chuo-ku, Tokyo 103-0027, Japan.}

\date{\today  \quad (Version $\beta$ -- for limited distribution)}

\begin{abstract}
	We prove a uniform version of Varadhan decomposition for shift-invariant 
	closed uniform forms 
	associated to large scale interacting systems
	on general crystal lattices.	
	In particular, this result includes the case of
	translation invariant processes on Euclidean lattices $\Z^d$ with finite range.
	Our result generalizes the result of \cite{BKS20} which was 
	valid for systems on transferable graphs.
	In subsequent research \cite{BS22L2}, we will use the result of this article
	to prove Varadhan's decomposition
	of closed $L^2$-forms for large scale interacting systems on general crystal lattices.
\end{abstract}

\maketitle

\tableofcontents

%
%
%
\section{Introduction}\label{sec: intro}
%
%
%

\emph{Varadhan's Decomposition of shift-invariant closed $L^2$-forms}
has played an important role in the strategy 
of proof of the hydrodynamic limit for 
\emph{non-gradient} models \cite{KL99}*{\S7}.
See the introduction of \cite{BKS20} for a short overview of hydrodynamic limits
and Varadhan's strategy.
In this article, we prove a \emph{uniform version} of such decomposition
for large scale interacting systems on a general \emph{crystal lattice}.
In particular, our result is valid for various translation invariant
processes on Euclidean lattices $\Z^d$ with finite range.
This generalizes our result in \cite{BKS20} for systems on locales which are transferable
or weakly transferable.  
The simple case of the exclusion process was previously announced in \cite{KS16}.
However, whilst the rough outline of proof remains similar, many new ideas, such as the 
origin of the conserved quantities and the notion for an interaction to be irreducibly quantified,
as well as the strategy
to reduce the proof of the key theorem \cref{thm: 1} to the 
case of essentially Euclidean crystal lattices 
were required to grasp and generalize many of the arguments in \cite{KS16} which a priori relied on properties
specific to the exclusion process.
Due to such conceptual advancements, the result of this article is valid for far general interactions, including models with multiple conserved quantities.
Although there has been some research concerning the hydrodynamic 
limit on crystal lattices (\cite{T12}, more recently \cites{F20,G20}), 
the known cases for interacting systems on non-Euclidean lattices are all for the gradient case
-- there are no known examples in the non-gradient case.

A \emph{crystal lattice}, or a \emph{topological crystal}
in the terminology of \cite{Sun13}*{\S6.2}, 
is any connected locally finite multi-graph $\cX=(X,E)$ for the set of vertices $X$
and  set of edges $E$, with a free action of a finitely generated abelian group $G$
of rank $>0$ such that the quotient graph $\cX/G=(X/G,E/G)$ is a strictly symmetric finite multi-graph.
See \S\ref{subsec: graphs} for terminology concerning graphs.
The dimension $\dim\cX$ is defined to be the rank of $G$.
The \emph{Euclidean lattices}, \emph{triangular lattice}, \emph{hexagonal lattice}, 
and the \emph{diamond lattice} are all examples of crystal lattices.

We consider a non-empty set $S$ which we call the \emph{set of states}, and $S^X\coloneqq\prod_{x\in X}S$
the \emph{configurations} on $X$.  We fix a map $\phi\colon S\times S\rightarrow S\times S$
called an \emph{interaction} satisfying a certain symmetry condition which specifies the
interaction of states on adjacent vertices.
Models such at the \emph{generalized exclusion process} or the
\emph{multi-species exclusion process} may be described using such interaction
(see \cite{BKS20}*{Example 2.18} for other examples).
One  important invariant of the interaction is the \emph{conserved quantity},
which is a normalized map $\xi\colon S\rightarrow\R$ satisfying  
\[
	\xi(s'_1)+\xi(s'_2)=\xi(s_1)+\xi(s_2),
\]
where $(s'_1,s'_2)=\phi(s_1,s_2)$ for any $(s_1,s_2)\in S\times S$.
We denote by $\C{S}$ the $\R$-linear space of conserved quantities.
Then $c_\phi\coloneqq\dim_\R\C{S}$ gives the number of independent  
conserved quantities of our large scale interacting system.

Via the interaction, for any edge $e\in E$, a configuration $\e=(\e_x)\in S^X$ transitions
to a configuration  $\e^e=(\e^e_x)\in S^X$ whose component 
differs only on the origin and target of the edge $e$.
Such transitions give a structure of a multi-graph on $S^X$.
We denote by $C(S^X)$ and $C^1(S^X)$ the space of $\R$-valued functions and forms on $S^X$,
where  a form is a system of functions $\omega=(\omega_e)_{e\in E}$ 
satisfying some minor conditions, where  $\omega_e$ for each $e\in E$ is a function on $S^X$.

We may  define the differential $\partial_{\cX}\colon C(S^X)\rightarrow C^1(S^X)$
by $\partial_{\cX} f\coloneqq(\nabe f)_{e\in E}$, where
\[
	\nabe f(\e)\coloneqq f(\e^e)-f(\e),\qquad \forall\e\in S^X.
\]
Functions and forms are \emph{local} if it depends only on states at a finite number of 
vertices of $X$.  In our previous work \cite{BKS20} we extended the notion of 
local functions to that of \emph{uniform functions}, which allow for infinite sums of 
normalized local functions.  Let $C_\unif(S^X)$ and $C^1_\unif(S^X)$ be the space
of uniform functions and forms.  If we denote by $Z^1_\unif(S^X)$
the space of closed uniform form, consisting of uniform forms whose integration
over a path depends only on the origin and the target of the path,
then the differential induces the differential
\[
	\partial_{\cX}\colon C_\unif(S^X)\rightarrow Z^1_\unif(S^X).
\]
The group $G$ acts naturally on $C_\unif(S^X)$ and $Z^1_\unif(S^X)$, and
the action is compatible with the differential $\partial_{\cX}$.
We let
\begin{align*}
	\cC&\coloneqq Z^1_\unif(S^X)^G,&
	\cE&\coloneqq \partial_{\cX}(C_\unif(S^X)^G)
\end{align*}
the space of shift-invariant closed uniform forms and
shift-invariant uniform exact forms.

We say that an interaction is \emph{irreducibly quantified}  
if the following condition holds for any finite $\cX=(X,E)$: 
If two configurations $\e=(\e_x),\e'=(\e'_x)\in S^X$ 
satisfy $\sum_{x\in X}\xi(\e_x)=\sum_{x\in X}\xi(\e'_x)$ for any
conserved quantity $\xi\in\C{S}$, then there exists a path from $\e$ to $\e'$
in $S^X$.  We say that an interaction is \emph{simple}, if $c_\phi=1$,
and for any non-zero $\xi\in\C{S}$, the  monoid generated by 
$\xi(S)\subset\R$ is isomorphic to $\N$ or $\Z$.

Our main theorem is as follows.

\begin{theorem}[=\cref{thm: main}]
	Let $\cX=(X,E)$ be a crystal lattice of dimension $d=\dim\cX$,
	and consider an interaction which is irreducibly quantified.
	We assume in addition that the interaction is simple if $d=1$.
	Then we have a decomposition
	\[
		\cC=\cE\oplus\partial_{\cX}\cV,
	\]
	where $\partial_{\cX}\cV\subset\cC$ is a certain subspace given in \cref{lem: cd}
	of dimension $c_\phi d$, where $c_\phi=\dim_\R\C{S}$.
\end{theorem}

\cref{thm: main} is a generalization of \cite{BKS20}*{Theorem 5,17} to the 
case of an arbitrary crystal lattice.  In a subsequent article \cite{BS22L2}, will use this result
to prove the decomposition of shift-invariant closed $L^2$-forms
on models over a general  crystal lattice, generalizing 
the result of \cite{BS21L2} for the case of the Euclidean lattice $\Z^d=(\Z^d,\bbE^d)$.
Such result will play a crucial role in the proof of the hydrodynamic limit
for large scale interacting systems on crystal lattices.
In particular, we
expect this result
would allow for the proof of the hydrodynamic limit for
translation invariant processes on $\Z^d$ with finite range, extending 
results which were up until now
known only for nearest neighbor interactions.

%
%
%
\section{Crystal Lattices}\label{sec: review}
%
%
%

In this section, we review some terminology concerning multi-graphs and crystal lattices.
We will then introduce the notion of uniform functions and forms
on a configuration space with transition structure.

%
\subsection{Multigraphs and Crystal Lattices}\label{subsec: graphs}
%

In this subsection, we introduce terminology concerning multi-graphs and crystal lattices.
Our main reference is \cite{Sun13}*{\S3.1}.
Since crystal lattices are in general multi-graphs instead of graphs,
we extend our previous works \cites{BKS20,BS21,BS21L2} to the case
when the underlying space is a multi-graph, which allows for multiple edges between 
two vertices.   
A \emph{graph} in \cite{Sun13}*{\S3.1} corresponds to a 
\emph{strictly symmetric multi-graph} in our terminology.
We will always assume that a multi-graph is symmetric.

\begin{definition}
	We define a \emph{symmetric directed multi-graph}, which we refer to simply as a \emph{multi-graph},
	to be the quintuple $\cX=(X,E_X,o,t,\iota)$, where $X$ is a 
	set which we call the \emph{set of vertices}, $E_X$ is a set which we call the \emph{set of  edges},
	$o$ and $t$ are maps $o\colon E_X\rightarrow X$ and $t\colon E_X\rightarrow X$ which 
	we call the \emph{origin} and \emph{target}, and $\iota\colon E_X\rightarrow E_X$, $e\mapsto\bar e$ is a bijection
	which we call the \emph{inversion} satisfying $o(\bar e)=t(e)$ 
	and $t(\bar e)=o(e)$.  
	We say that $\cX$ is \emph{strictly symmetric}, if $\bar e\neq e$ for any $e\in E_X$.
\end{definition}

We will often denote $E_X$ by $E$ for simplicity.
Suppose $\cX=(X,E,o,t,\iota)$ is a  multi-graph.  For any  edge
$e\in E$, we call $o(e)$ and $t(e)$ the origin and target of $e$.
In literature, the map $(o,t)\colon E\rightarrow X\times X$ is referred to as an \emph{incidence map}.
We will often denote $(X,E,o,t,\iota)$ simply as $(X,E)$.
We say that a multi-graph $\cX=(X,E)$
\emph{is locally finite} if for any $x\in X$, 
the set $E_x\coloneqq\{e\in E\mid o(e)=x\}$ is finite for any $x\in X$,
\emph{is finite} if its set of vertices $X$ and the set of edges $E$ are both finite,
and \emph{is infinite} if it is not finite.
If a multi-graph $(X,E)$ is locally finite, then it is finite if and only if the set of vertices $X$ is finite.

We define a \emph{path} on a multi-graph 
$\cX=(X,E)$ to be a finite sequence $\vec p\coloneqq(e^1,\ldots,e^N)$
of edges satisfying $t(e^i)=o(e^{i+1})$ for any $0<i<N$.  
We say that $\vec p$ is a path from $o(\vec p)\coloneqq o(e^1)$ to $t(\vec p)\coloneqq t(e^N)$ of 
length $\len(\vec p)\coloneqq N$.  
For any $x,x'\in X$, we let $P(x,x')$ the set of paths from $x$ to $x'$.
We define the \emph{distance} from $x$ to $x'$ by
\[
	d_{\cX}(x,x')\coloneqq\inf_{\vec p\in P(x,x')}\len(\vec p)
\]
if $P(x,x')\neq\emptyset$ and $d_{\cX}(x,x')\coloneqq\infty$ otherwise.
For any $Y\subset X$, we say that $Y$ is \emph{connected}, if for any $y,y'\in Y$,
there exists a path from $y$ to $y'$ consisting of edges whose origin and target are all elements of $Y$.
If $X$ is connected, then we say that the multi-graph $\cX=(X,E)$ is connected.

Any $e\in E$ such that $o(e)=t(e)$ is referred to as a \emph{loop}.
A multi-graph $\cX=(X,E)$ is a \emph{graph}, if the incidence
map $E\rightarrow X\times X$ is injective.
For the case of graphs, we will identify $E$ with its image in $X\times X$ so that $E\subset X\times X$,
and the maps $o$ and $t$ are simply the first and second projections.
Using this convention,
the inversion of a graph $\cX$ coincides with the bijection given by exchanging the components of $E\subset X\times X$, and any loop $e\in E$ satisfies $\bar e=e$.
A \emph{simple graph} is a graph without any loops.
Graphs are referred to as a combinatorial graph in \cite{Sun13}. 
We may associate to a multigraph $\cX=(X,E)$ a graph $\cX'=(X,E')$ by taking $E'\subset X\times X$
to be the image of the incidence map, which we refer to a \emph{graph associated to $\cX$}.
The \emph{simple graph $\cX''$ associated to $\cX$}
is the simple graph obtained by removing the loops from $\cX'$.

In this article, we allow for the locale underlying the large scale interacting system to be a multi-graph,
which may not necessarily be simple.  Hence we extend the definition of locales given in \cite{BKS20} to multi-graphs as follows.

\begin{definition}\label{def: locale}
	We let the \emph{locale} to be any connected locally finite multi-graph.
\end{definition}

Next, we define the notion of morphisms and automorphisms of multi-graphs.

\begin{definition}
	Let $\cX=(X,E_X)$ and $\cY=(Y,E_Y)$ be multi-graphs.
	A morphism of a multi-graph $u\colon\cX\rightarrow\cY$
	is a pair $(u,u_E)$ consisting of
	maps of sets $u\colon X\rightarrow Y$ and $u_E\colon E_X\rightarrow E_Y$ compatible
	with the incidence map and the inversion.  In other words, we have commutative diagrams
	\[
		\xymatrix{
			E_X \ar[r]^{u_E} \ar[d]_{(o,t)}&  E_Y\ar[d]_{(o,t)}  &    E_X\ar[r]^{u_E} \ar[d]_\iota  & E_Y\ar[d]_\iota\\
			X\times X\ar[r]^{u\times u}&Y\times Y               &     E_X\ar[r]^{u_E}   & E_Y,
		}
	\]
	where the vertical maps are the incidence and inversion maps.
	Note that if $\cX$ is a graph, then the map $u_E$ is determined uniquely by $u$.
\end{definition}

An \emph{automorphism} of a multi-graph $\cX=(X,E_X)$ is a morphism $(\sigma,\sigma_E)\colon\cX\rightarrow\cX$
such that the induced maps $\sigma\colon X\rightarrow X$ and $\sigma_E\colon E_X\rightarrow E_X$ are bijective.
We denote by $\Aut(\cX)$ the set of automorphisms of $\cX$.  The set 
$\Aut(\cX)$ has a structure of a group whose operation
is the composition of automorphisms.

\begin{definition}
	Let $\cX=(X,E_X)$ and $\cY=(Y,E_Y)$ be multi-graphs.
	We say that a morphism $u\colon\cX\rightarrow\cY$ of multi-graphs
	is a \emph{covering}, if $u$ and $u_E$ are both surjective, and
	for any $x\in X$, if we let $E_x\coloneqq \{ e\in E_X \mid o(e)=x\}$
	and  $E_{u(x)}\coloneqq \{ e\in E_Y \mid o(e)=u(x)\}$, then the map $E_x\rightarrow E_{u(x)}$
	induced by $u_E$ is bijective.
\end{definition}

If $u\colon\cX\rightarrow\cY$ is a covering, then $\cX$ is locally finite if and only if $\cY$
is locally finite.
For any multi-graph $\cX=(X,E)$, an action of a group $G$ on $\cX$
is defined to be a group homomorphism $G\rightarrow\Aut(\cX)$.
We define the \emph{quotient multi-graph $\cX/G$} of $\cX$ for the action of $G$ 
as follows.
Let $X/G$ and $E/G$ be the set of orbits which we refer also to
as a \emph{quotient} of $X$ and $E$ with respect to the action of $G$.
For any $e\in E$ and $\sigma\in G$, we have 
\[
	(o(\sigma_E(e)), t(\sigma_E(e)))= (\sigma(o(e)),\sigma(t(e)),
\]
hence the origin and target maps induces maps $o,t\colon E/G\rightarrow X/G$ on the quotient.
\begin{proposition}\label{prop: covering}
	The data $\cX/G=(X/G,E/G,o,t,\iota)$ is again a multi-graph, which we call the quotient of
	$\cX$ with respect to the action of $G$.    
	If the action of $G$ on $X$ is free, 
	in other words, if $\sigma(x)=x$ for some $x\in X$ implies that $\sigma=\id$,
	then the natural map $\cX\rightarrow\cX/G$
	is a covering.
\end{proposition}

\begin{proof}
	The fact that $\cX/G$ is a multi-graph follows immediately from the definition.
	Suppose the action of $G$ on $\cX$ is free.
	Let $x\in X$ and denote by $x_0$ its image in $X/G$.
	Then we have a map
	\begin{equation}\label{eq: map}
		E_{x}=\{  e\in E \mid  o(e)=x\} \rightarrow E_{x_0}=\{  e_0\in E/G \mid  o(e_0)=x_0\}.
	\end{equation}
	If $e_0\in E_{x_0}$, then there exists a representative $e'\in E$ such that $o(e')=\sigma(x)$
	for some $\sigma\in G$.  If we let $e\coloneqq\sigma^{-1}_E(e')$, then $e\in E_x$ 
	and maps to $e_0$ in $E_{x_0}$.  This shows that \eqref{eq: map}
	is surjective.  Suppose now that we have $e,e'\in E$ such that $o(e)=o(e')=x$,
	and suppose $e$ and $e'$ coincides in $E/G$.  Then $e=\sigma_E(e')$ for some $\sigma\in G$.	
	Then compatibility of the action of $G$ with $o$ shows that $x=o(e)=o(\sigma_E(e'))=\sigma(o(e'))=\sigma(x)$.
	Since the action of $G$ on $X$ is free, we see that $\sigma=\id$, which shows that $e=e'$
	hence that \eqref{eq: map} is injective as desired.
\end{proof}

We see that any path in $\cX$ defines a path in $\cX/G$, hence if $\cX$ is connected, then $\cX/G$ is also connected.
For a covering $u\colon\cX\rightarrow\cY$, we say that an automorphism $\sigma$ of $\cX$ is a 
\emph{covering transformation} if $u\circ\sigma=u$, in other words, that we have
a commutative diagram
\[
	\xymatrix{\cX\ar[rr]^\sigma\ar[dr]_u&  &  \cX \ar[dl]^u\\  &  \cY&}.
\] 
 We denote by $\Aut(\cX/\cY)$ the group of covering transformations
of $\cX$ over $\cY$.  By \cite{Sun13}*{Theorem 5.2}, if we let $\cY=\cX/G$, then $G\cong\Aut(\cX/\cY)$.

We define a crystal lattice as follows.

\begin{definition}
	\begin{enumerate}
		\item We define a \emph{crystal} to be a strictly symmetric finite multi-graph.
		This coincides with what is referred to as a \emph{finite graph} in \cite{Sun13}.
		\item We say that a connected multi-graph $\cX=(X,E)$ is a \emph{crystal lattice}, if it has
		a free action by a finitely generated abelian group $G$ of rank $>0$, such that $\cX_0\coloneqq\cX/G$ is a 
		crystal.  
	\end{enumerate}
\end{definition}
	
Note that a crystal lattice is locally finite since $\cX\rightarrow\cX/G$ is a covering by \cref{prop: covering}.
Hence in particular, a crystal lattice is a locale in the sense of Definition \ref{def: locale}.
Since $\rank G>0$, a crystal lattice is an infinite milti-graph.
We refer to the crystal $\cX_0=X/G$ as the \emph{base graph} or the \emph{seed crystal} 
of the crystal lattice $\cX$.  Note that a seed crystal is connected since we have assumed that $\cX$
is connected.
We define the dimension of $\cX$ denoted $\dim\cX$ to be the rank of the abelian group $G$.

\begin{example}
	\begin{enumerate}
		\item The most typical example of a crystal
		lattice is the \emph{Euclidean lattice} $\cX=(\Z^d,\bbE^d)$ for any integer $d\geq1$,
		where $\Z^d$ is the $d$-fold product of $\Z$ and
		\[
			\bbE^d\coloneqq\{(x,y)\in\Z^d\times\Z^d\mid\abs{x-y}=1\},
		\]
		where $\abs{x}\coloneqq\sum_{j=1}^d\abs{x_j}$ for $x=(x_j)\in\Z^d$,
		with the action of $G=\Z^d$ given by translation.
		The quotient $\cX_0=\cX/G$ is the finite multi-graph $\cX_0=(X_0,E_0,o,t)$ 
		such that $X_0\coloneqq\{x_0\}$ 
	consists of a single vertex, $E_0\coloneqq\{e_j,\bar e_j\mid j=1,\ldots,d\}$, and $o,t\colon E_0\rightarrow X_0$
	are both constant maps with image $x_0$.
		\item Let $\cX_0$ be a connected crystal.  Then by \cite{Sun13}*{Theorem 6.6},
		the maximal abelian covering $\cX^\ab$ of $\cX_0$ is a connected 
		simple graph.
		The group of covering transformations $\Aut(\cX^\ab/\cX_0)$ of $\cX^\ab$ over $\cX_0$
		is isomorphic to the maximal abelian quotient $\pi_1^\ab(\cX_0)$
		of the fundamental group of $\cX_0$, which shows that $\cX/\pi_1^\ab(\cX_0)=\cX_0$,
		hence that $\cX$ is a crystal latttice
		for the action of $\pi_1^\ab(\cX_0)$.
		\item Let $\cX$ be a crystal lattice for the action by a finitely generated abelian group 
		$G$.
		If we let $\cX^\ab$ be the maximal abelian covering of the seed crystal
		$\cX_0\coloneqq\cX/G$,
		then we have $\cX=\cX^\ab/H$ for some subgroup $H\subset\pi_1^\ab(\cX_0)$,
		and we have $G\cong\pi_1^\ab(\cX_0)/H$.
	\end{enumerate}
	See \cref{example: more} for more examples of crystal lattices.
\end{example}

%
\subsection{Configuration Space and Uniform Functions}\label{subsec: CS}
%

In order to consider the large scale interacting system on a locale, we next consider the configuration
space with transition structure.

\begin{definition}
	We define an \emph{interaction} to be a pair $(S,\phi)$ consisting of a set $S$
	which we call the \emph{set of states} and a
	map
	$\phi\colon S\times S\rightarrow S\times S$ satisfying
	\begin{equation}\label{eq: interaction}
		\bar\phi\circ\phi(s_1,s_2)=(s_1,s_2)
	\end{equation}
	for any $(s_1,s_2)\in S\times S$ such that $\phi(s_1,s_2)\neq(s_1,s_2)$.
	Here, $\bar\phi\coloneqq\iota\circ\phi\circ\iota$ for the bijection 
	$\iota\colon S\times S\rightarrow S\times S$ obtained by exchanging the components of 
	$S\times S$.
\end{definition}

In what follows, we let $(S,\phi)$ be an interaction.
For any multi-graph $\cX=(X,E)$,
we define the \emph{configuration space} on $\cX$  by
\[
	S^X\coloneqq\prod_{x\in X}S.
\]
We call any $\e\in S^X$ a \emph{configuration}.
For any configuration $\e=(\e_x)\in  S^X$ and $e\in E$, we denote by $\e^e$
the element $\e^e=(\e^e_x)\in S^X$ such that 
$(\e^e_{o(e)},\e^e_{t(e)})\coloneqq\phi(\e_{o(e)}, \e_{t(e)})$ if $o(e)\neq t(e)$,
and $\e^e_x=\e_x$ if $x\neq o(e),t(e)$
or $x=o(e)=t(e)$.
Furthermore, we let
\[
	\Phi_X\coloneqq S^X \times E,
\]
and we define the maps $o,t\colon \Phi_X\rightarrow S^X\times S^X$ by
\begin{align}\label{eq: ot}
	o((\e,e))&=\e,&
	t((\e,e))&=\e^e
\end{align}
for any $(\e,e)\in\Phi_X$.  We have the following.

\begin{lemma}\label{ref: choice}
	For any $e\in E$ and $\e\in S^X$, if $\e^e\neq\e$, then $(\e^e)^{\bar e}=\e$.
\end{lemma}

\begin{proof}
	Let $e\in E$ and $\e=(\e_x)\in S^X$.
	Note that $(\e_{o(e)},\e_{t(e)})\in S\times S$.
	Suppose that $\e^e\neq\e$.
	Then $\phi(\e_{o(e)},\e_{t(e)})\neq(\e_{o(e)},\e_{t(e)})$.
	By definition of $(\e^e)^{\bar e}$, we see that
	\[
		((\e^e)^{\bar e}_{o(e)},(\e^e)^{\bar e}_{t(e)})=\bar\phi\circ\phi(\e_{o(e)},\e_{t(e)}).
	\]
	Our assertion now follows from \eqref{eq: interaction}.
\end{proof}

By \cref{ref: choice}, we have $(\e^e)^{\bar e}=\e$ if $\e^e\neq\e$
for any $(\eta,e)\in\Phi_X$.
We define the inversion $\iota\colon\Phi_X\rightarrow\Phi_X$ by
\begin{equation}\label{eq: iota}
	\iota((\e,e))\coloneqq
	\begin{cases}
		(\e,e)  & \e^e=\e\\
		(\e^e,\bar e)  & \e^e\neq\e
	\end{cases}
\end{equation}
for any $(\e,e)\in\Phi_X$.
Then by construction, if $\e^e=\e$, then $\iota\circ\iota(\e,e)=\iota(\e,e)=(\e,e)$
and if $\e^e\neq\e$, then $\iota\circ\iota(\e,e)=\iota(\e^e,\bar e)=((\e^e)^{\bar e},e)=(\e,e)$.
This shows that $\iota\colon\Phi_X\rightarrow\Phi_X$ is a bijection
such that if we denote $\bar\vp=\iota(\vp)$ for any $\vp\in\Phi_X$, then 
we have $o(\bar\vp)=t(\vp)$ and $t(\bar\vp)=o(\vp)$.
This shows that $(S^X,\Phi_X,o,t,\iota)$ is a multi-graph.

\begin{definition}\label{ref: symmetric}
	The multi-graph 
	$
		S^{\cX}_\phi\coloneqq(S^X,\Phi_X,o,t,\iota),
	$ 
	often denoted simply as $S^{\cX}_\phi=(S^X,\Phi_X)$,
	defines a multi-graph
	which we call the \emph{configuration space with transition structure} associated
	to the multi-graph $\cX$ and interaction $(S,\phi)$.
\end{definition}

\begin{remark}
	In our previous articles \cites{BKS20,BS21,BS21L2}, 
	we defined the configuration space with transition structure simply to be
	a graph and not a multi-graph. 
	We note that $S^{\cX}_\phi$ is in general not strictly symmetric.
\end{remark}

Let $\cX=(X,E)$ be a locally finite multi-graph, and let $(S^X,\Phi_X)$ be the configuration space
with transition structure associated to some interaction $(S,\phi)$.
For any $\La\subset X$, let $C(S^\La)$ be the $\R$-linear space of maps from $S^\La$ to $\R$.
For any $\La\subset\La'\subset X$,
the natural projection $S^{\La'}\rightarrow S^\La$ induces an injective $\R$-linear homomorphism
$C(S^\La)\hookrightarrow C(S^{\La'})$.  Denote by $\sI_X$ the set of finite subsets of $X$.
We define the space of \emph{local functions} $C_\loc(S^X)$ on $S^X$ by
\[
	C_\loc(S^X)\coloneqq\bigcup_{\La\in\sI_X}C(S^\La)\subset C(S^X).
\]
Next, we fix a point $*\in S$, which we call the \emph{base state}, and we let $S^X_*$ be the
configurations $\e=(\e_x)\in S^X$ such that $\e_x=*$ for all but finite $x\in X$.
For any $\La\in\sI_X$, the natural projection $S^X_*\rightarrow S^\La$ induces an injective
$\R$-linear homomorphism $C(S^\La)\hookrightarrow C(S^X_*)$, which induces an embedding
\[
	C_\loc(S^X)\hookrightarrow C(S^X_*),
\]
which is compatible with the projection $C(S^X)\rightarrow C(S^X_*)$.

\begin{proposition}[\cite{BKS20}*{Proposition 3.3}]\label{prop:BKS203.3}
	For any $\La\in\sI_X$, let 
	\[
		C_\La(S^X)\coloneqq \{f\in C(S^\La)\mid f(\e)=0  \text{ if $\e=(\e_x)$ satisfies $\eta_x=*$ for some $x\in \La$}\}.
	\]
	Then any $f\in C(S^X_*)$ has a unique expansion of the form
	\begin{equation}\label{eq: expansion}
		f=\sum_{\La\in\sI_X}f_\La,
	\end{equation}
	where $(f_\La)_{\La\in\sI_X}$ is a set of functions such that $f_\La\in C_\La(S^X)$ for any $\La\in\sI_X$.
\end{proposition}

Suppose we are given a set of functions $(f_\La)_{\La\in\sI_X}$
such that $f_\La\in C_\La(S^X)$ for any $\La\in\sI_X$.
Then for any $\e\in S^X_*$, we have $f_\La(\e)\neq 0$
for all but finite $\La\in\sI_X$.  This shows that the sum $\sum_{\La\in\sI_X}f_\La(\e)$ is well defined
for any $\e\in S^X_*$, hence defines a function on $S^X_*$ (see \cite{BKS20}*{Lemma 3.2} for details).
For any $R\geq 0$, we let $\sI_{\cX,R}$ be the subset of $\sI_X$ consisting of $\La\in\sI_X$ such that $\diam(\La)\leq R$, where $\diam(\La)\coloneqq\sup_{x,x'\in\La}d_{\cX}(x,x')$.
Since we have assumed that $\cX$ is locally finite, any $\La\subset X$ with finite diameter is a finite set.

In our previous article \cites{BKS20,BS21,BS21L2}, we assumed
that the underlying locale is a simple graph.
Multi-edges in our current usage do not typically have a big effect, and various assertions 
for the associated simple graph may naturally be extended to the case of multi-graphs.
For example, the distance $d_{\cX}$ on $X$ coincides for a multi-graph and its associated simple graph.

\begin{definition}\label{def: UL}
	We say that a function $f\in C(S^X_*)$ is a \emph{uniform function on $S^X$}, if there exists $R\geq 0$ such that
	\[
		f=\sum_{\La\in\sI_{\cX,R}}f_\La
	\]
	for some set of functions such that $f_\La\in C_\La(S^X)$ for any $\La\in\sI_{\cX,R}$.
	We denote by $C_\unif(S^X)$ the $\R$-linear space of uniform functions on $S^X$.
\end{definition}
The notion of uniformity a priori depends on the choice of $E$ for the multigraph $\cX=(X,E)$.
We will later show that the notion is independent of the choice of $E$ for crystal lattices.
The most basic example of uniform functions is given by the conserved quantites.

\begin{definition}
	We say that a map $\xi\colon S\rightarrow\R$
	is a \emph{conserved quantity} for the pair $(S,\phi)$, if $\xi(*)=0$ for the base state $*\in S$,
	and
	\[
		\xi(s_1)+\xi(s_2)=\xi(s'_1)+\xi(s'_2),
	\]
	where $(s_1',s_2')=\phi(s_1,s_2)$ for any $(s_1,s_2)\in S\times S$.
\end{definition}

Let $\xi\in\C{S}$.  For any $x\in X$, 
let $\xi_x$ be the function in $C(S^{\{x\}})\subset C(S^X)$ such that
$\xi_x(\e)=\xi(\e_x)$ for any $\e=(\e_x)\in S^X$.
If we let $\xi_X\coloneqq\sum_{x\in X}\xi_x\in C(S^X)$,
then $\xi_X$ is a uniform function in $C_\unif(S^X)$.
The correspondence $\xi\mapsto\xi_X$ gives an injective
$\R$-linear homomorphism $\C{S}\hookrightarrow C_\unif(S^X)$.

We use the conserved quantities to define the notion that an interaction is irreducibly quantified.
Let $(\La,E_\La)$ be a finite locale.
In this case, we have $\xi_\La=\sum_{x\in\La}\xi_x\in C(S^\La)$.
Consider the canonical map
\begin{equation}\label{eq: bsxila}
	\bsxi_{\!\La}\colon S^\La\rightarrow\Hom_\R(\C{S},\R), \qquad
	\e\mapsto (\xi \mapsto\xi_\La(\e)).
\end{equation}
Note that a choice of basis of $\C{S}$ gives an isomorphism $\Hom_\R(\C{S},\R)\cong\R^{c_\phi}$.
For any $\e,\e'\in S^\La$, we have $\bsxi_{\!\La}(\e)=\bsxi_{\!\La}(\e ')$ if and only if $\xi_\La(\e)=\xi_\La(\e')$
for any conserved quantity $\xi\in\C{S}$.

\begin{definition}
	We say that an interaction $(S,\phi)$ is \emph{irreducibly quantified}, if for any finite locale 
	$(\La,E_\La)$ and $\e,\e'\in S^\La$ such that $\bsxi_{\!\La}(\e)=\bsxi_{\!\La}(\e')$,
	there exists a path in $S^\La$ from $\e$ to $\e'$.
\end{definition}

Irreducibly quantified is an important assumption for our main results.

%
\subsection{Forms and Uniform Forms}\label{subsec: forms}
%

In this subsection, we review the notion of forms and uniform forms.
We keep the notation of \S\ref{subsec: CS}.
Let 
\[
	\Phi_X^*\coloneqq\{\vp\in\Phi_X  \mid o(\vp),t(\vp)\in S^X_*\},
\]
so that $(S^X_*,\Phi^*_X,o,t,\iota)$ for maps $o,t,\iota$ as in \eqref{eq: ot} and \eqref{eq: iota} is 
again a multi-graph.
We let
\[
	C^1(S^X_*)\coloneqq\Map^\alt(\Phi_X^*,\R)
\]	
be the $\R$-linear space of maps $\omega^*\colon\Phi_X\rightarrow\R$, satisfying the condition
that $\omega(\bar\vp)=-\omega(\vp)$ for any $\vp\in\Phi_X^*$.
We call such $\omega$ a \emph{form} on $S^X_*$.
We have a differential
\[
	\partial_{\cX}\colon C(S^X_*)\rightarrow C^1(S^X_*)
\]
defined by $\partial_{\cX} f(\vp)\coloneqq f(\e^e)-f(\e)$ for any $\vp=(\e,e)\in\Phi^*_X$.

For any $e\in E$, a form $\omega\in C^1(S^X_*)$ defines a map $\omega_e\colon S^X_*\rightarrow\R$
by $\omega_e(\e)\coloneqq\omega((\e,e))$, which gives an inclusion
\begin{equation}\label{eq: inclusion}
	C^1(S^X_*)\hookrightarrow\prod_{e\in E}C(S^X_*), \qquad\omega\mapsto(\omega_e)_{e\in E}.
\end{equation}
A system $(\omega_e)_{e\in E}$ is a form if and only if
\begin{equation}\label{eq: 1}
	\omega_e(\e)=
	\begin{cases} 
	-\omega_{\bar e}(\e^e)& \e^e\neq\e,\\
		0&\e^e=\e.	
	\end{cases}
\end{equation}
If $f\in C(S^X_*)$, then $\partial_{\cX} f=(\nabe f)_{e\in E}$, where $\nabe f$ is the function in $C(S^X_*)$
given by $\nabe f(\e)\coloneqq f(\e^e)-f(\e)$ for any $\e\in S^X_*$.

For any path $\vec\gamma=(\vp_1,\ldots,\vp_N)$ in $S^X_*$, we define the integral of 
$\omega\in C^1(S^X_*)$
with respect to $\vec\gamma$ by
\[
	\int_{\vec\gamma}\omega\coloneqq\sum_{i=1}^N \omega(\vp_i).
\]
We say that a form $\omega\in C^1(S^X_*)$ is \emph{closed}, if
$\int_{\vec\gamma}\omega=0$ for any closed path $\vec\gamma\in S^X_*$.
We denote by $Z^1(S^X_*)$ the space of closed forms.
As in \cite{BKS20}*{Lemma 2.15 and 2.16}, we may prove the following.

\begin{lemma}\label{lem: exact}
	A form $\omega\in C^1(S^X_*)$ is closed if and only if it is exact, in other words,
	if there exists $f\in C(S^X_*)$ such that $\partial_{\cX} f=\omega$.
	In particular, the differential $\partial_{\cX}$ gives an isomorphism 
	$C(S^X_*)/\Ker\partial_{\cX}\cong Z^1(S^X_*)$.
\end{lemma}

For any $R\geq0$, we define the space of $R$-uniform forms by
\[
	C^1_R(S^X)\coloneqq C^1(S^X_*)\cap\prod_{e\in E}C\bigl(S^{B(o(e),R)}\bigr),
\]
where $B(o(e),R)\coloneqq\{ x\in X\mid d_{\cX}(o(e),x)\leq R\}$ is the ball in $X$ with center $o(e)$ and radius $R$.
Note that $B(o(e),R)$ is finite since we have assumed that the multi-graph $\cX$ is locally finite.
We define the $\R$-linear space of uniform forms on $S^X$ by
\[
	C^1_\unif(S^X)\coloneqq\bigcup_{R\geq 0}C^1_R(S^X),
\]
and the $\R$-linear space of closed uniform forms on $S^X$ by
\[
	Z^1_\unif(S^X)\coloneqq C^1_\unif(S^X)\cap Z^1(S^X).
\]
The differential induces an $\R$-linear map
\begin{equation}\label{eq: differential}
	\partial_{\cX}\colon C_\unif(S^X)\rightarrow Z^1_\unif(S^X)
\end{equation}
(see \cite{BKS20}*{Lemma 5.1}).
The key in the proof of our main result is the investigation of the kernel and the cokernel
of this differential.  The kernel may be expressed via the space of conserved quantities
as stated in \cref{thm: old} below,
when the interaction is irreducibly quantified.

Let $\star\in S^X_*$ be the configuration whose components are all at base state,
and we let $C^0_\unif(S^X_*)\coloneqq\{f\in C_\unif(S^X)\mid f(\star)=0\}$.
By construction, we have $\xi_X\in C^0_\unif(S^X)$ for any $\xi\in\C{S}$.
Since a conserved quantity is preserved with respect to interactions, for any $\e\in S^X_*$,
we have $\xi_X(\e^e)=\xi_X(\e)$ for any $e\in E$.  This shows that we have 
$\partial_{\cX}\xi_X=(\nabe\xi_X)_{e\in E}=0$ in $Z^1_\unif(S^X)$, hence we have $\xi_X\in\Ker\partial_{\cX}$.
One of the main results of \cite{BKS20} is the following, which will be used 
in the proof of our main theorem.

\begin{theorem}\label{thm: old}
	Let $\cX=(X,E)$ be a \emph{locale}, in other words, a locally finite multi-graph
	which is connected.
	Suppose in addition that the interaction $(S,\phi)$ is irreducibly quantified.
	Then the homomorphism $\C{S}\hookrightarrow C^0_\unif(S^X)$ induces an isomorphism
	\[
		\C{S}\cong \Ker\partial_{\cX}\subset C^0_\unif(S^X)
	\]
	for the differential $\partial_{\cX}\colon C^0_\unif(S^X)\rightarrow Z^1_\unif(S^X)$.
\end{theorem}

For the rest of this subsection, we prepare some lemmas concerning functions whose differentials are
uniform.   Let $\cX$ be a locally finite multi-graph, and consider an interaction $(S,\phi)$ 
which is irreducibly quantified. 
For any $f\in C(S^X_*)$ with expansion $f=\sum_{\La''\in\sI_X}f_{\La''}$,
we let
\[
	\iota^\La f\coloneqq\sum_{\La''\subset\La}f_{\La''}
\]
for $\La\in\sI_X$, which gives a function in $C(S^\La)$.
If $\e\in S^X_*$, we have
\[
	\iota^\La f(\e)=f(\e|_\La)
\]
for any $\La\in\sI_X$, where $\e|_\La\in S^X_*$ is the configuration whose component 
coincides with that of $\eta$ for $x\in\La$ and is at base state for $x\not\in\La$.

For any $\La\subset X$ and $R\geq0$, we let $B(\La,R)\coloneqq\bigcup_{x\in\La}B(x,R)$.

\begin{lemma}\label{lem: dependence}
	Let $f\in C(S^X_*)$ and suppose $\partial_{\cX}f\in C^1_R(S^X)$ for some $R>0$.
	Let $Y\subset X$ be a connected subset.
	Then for any $W\subset X$ containing $B(Y,R)$, we see that the value
	\[
		(f-\iota^{W} f)(\e)
	\]
	for $\e\in S^X_*$ depends only on $\bsxi_Y(\e)$ and $\e|_{X\setminus Y}$.
\end{lemma}

\begin{proof}
	Consider $\e',\e''\in S^X_*$ which coincide outside $Y$
	and satisfies $\bsxi_Y(\e')=\bsxi_Y(\e'')$.
	Since the interaction is irreducibly quantified,
	there exists a path 
	$\vec\gamma|_{S^Y}=(\vp_1|_{S^Y},\ldots,\vp_N|_{S^Y})$ from $\e'|_{S^Y}$ to $\e''|_{S^Y}$
	in $S^Y$.  We let $e_1,\ldots,e_N\in E_Y$ be
	such that $\vp_i|_{S^Y}=(\tilde\e_i,e_i)$, where $\tilde\e_1=\eta'|_{S^Y}$
	and $\tilde\eta_{i+1}=\tilde\eta_i^{e_i}$ for $i=1,\ldots,N-1$.
	If we let $\vp_i\coloneqq(\e_i,e_i)$, where $\e_1=\eta'$
	and $\eta_{i+1}=\eta_i^{e_i}$ for $i=1,\ldots,N-1$, then $\vec\gamma\coloneqq(\vp_1,\ldots,\vp_N)$
	gives a path from $\e'$ to $\e''$.
	Then we have
	\begin{align*}
		f(\e'')-f(\e')&=\sum_{i=1}^N\nabla_{e_i}f(\e_i), &
		f(\e''|_{W})-  f(\e'|_{W})&
		=\sum_{i=1}^N\nabla_{e_i}f(\e_i|_{W}).
	\end{align*}
	By our assumption, we have 
	$\nabe f\in C\bigl(S^{B(o(e),R)}\bigr)\subset C(S^W_*)$ for any $e\in E_Y$.
	This shows that
	\[
		\nabe f(\e|_{W})=\nabe f(\e)
	\]
	for any $e\in E_Y$ and $\e\in S^X_*$, hence
	$
		f(\e'')-f(\e')=f(\e''|_{W})- f(\e'|_{W}),
	$
	which proves that
	\[
		f(\e'')-f(\e''|_{W})=f(\e')- f(\e'|_{W})
	\]
	as desired.
\end{proof}

For any $\La\in\sI_X$,
let $\bsxi_{\!\La}\colon S^\La\rightarrow\Hom_\R(\C{S},\R)$ be the canonical map 
given in \eqref{eq: bsxila}.  The image of $S^\La$ by $\bsxi$ depends only on the order of
$\abs{\La}$, so we define $\cM_{\abs{\La}}\coloneqq\bsxi_{\!\La}(S^\La)$ for any $\La\in\sI_X$.
\begin{proposition}\label{prop: pairing}
	Let $f\in C(S^X_*)$ such that $\partial_{\cX}f\in C^1_R(S^X)$ for some $R>0$.
	Then for any finite connected $\La_1,\La_2\subset X$ such that $d_{\cX}(\La_1,\La_2)>R$,
	there exists a function $h^{\La_1,\La_2}_f\colon\cM_{\abs{\La_1}}\times\cM_{\abs{\La_2}}\rightarrow\R$ 
	such that
	\begin{equation}\label{eq: pair}
		\iota^{\La_1\cup \La_2}f(\e)-\iota^{\La_1} f(\e)-\iota^{\La_2}f(\e)=h^{\La_1,\La_2}_f(\bsxi_{\!\La_1}(\e),\bsxi_{\!\La_2}(\e))
	\end{equation}
	for all $\e\in S^X_*$.
\end{proposition}

\begin{proof}
	It is sufficient to prove the statement for $\eta=\eta|_{\La_1\cup \La_2}$.
	Let $Y=\La_1$ and $W=B(\La_1,R)$.
	Since $\La_2\cap W=\emptyset$, we have $\eta|_W=\eta|_{\La_1}$ and
	\[
		(f-\iota^W f)(\eta|_{\La_1\cup \La_2})=f(\eta|_{\La_1\cup \La_2})-f(\eta|_{\La_1})
		=\iota^{\La_1\cup \La_2}f(\e)-\iota^{\La_1} f(\e).
	\]
	This equality and \cref{lem: dependence} show that the left hand side of \eqref{eq: pair}
	depends only on $\bsxi_{\!\La_1}(\eta)$ and $\eta|_{\La_2}$.
	By replacing the roles of $\La_1$ and $\La_2$, we see that 
	the left hand side of \eqref{eq: pair} depends only on 
	$\bsxi_{\!\La_1}(\eta)\in\cM_{\abs{\La_1}}$ and $\bsxi_{\!\La_2}(\eta)\in\cM_{\abs{\La_2}}$.
\end{proof}

By definition, for any finite connected $\La_1,\La_2\subset X$ and
$\alpha\in\cM_{\abs{\La_1}}$ and $\beta\in\cM_{\abs{\La_2}}$,
we have 
\begin{equation}\label{eq: sym}
	h^{\La_1,\La_2}_f(\alpha,\beta)=h^{\La_2,\La_1}_f(\beta,\alpha).
\end{equation}

%
%
%
\section{Criterion for Uniformity}\label{sec: criterion}
%
%
%

In this section, we introduce the notion of an essentially Euclidean crystal lattice,
and to prove a criterion \cref{prop: criterion} for ensuring that a function on $S^X_*$ 
for such crystal lattice is uniform.  
This criterion will be used in the next section
to prove \cref{thm: 0} concerning the surjectivity of the differential
from the space of uniform functions to closed uniform forms.
In \S\ref{subsec: EE}, we will give examples of essentially Euclidean crystal
lattices.

%
\subsection{Pairings for Essentially Euclidean Crystal Lattices}\label{subsec: EECL}
%

In what follows, let $\cX=(X,E)$ be a crystal lattice for the action of a finitely generated abelian group $G$.
In this subsection, we introduce the notion of 
\emph{essentially Euclidean crystal lattices}, 
and construct certain pairings associated to each function $f\in C(S^X_*)$ whose differential is uniform.
We first introduce a Euclidean type coordinate on $\cX$, which we call a \emph{block coordinate}.

We let $F \subset G$ be a free subgroup of finite index in $G$.  
Note that in this case, the rank of $F$ coincides with that of $G$.
We say that 
$\La_F$ is a  \emph{fundamental domain} of $\cX$ for the action of $F$, if $\La_F$
is connected, and $\La_F$ gives a representative of the orbits of $X$ with respect to the action of $F$,
so that for any $x\in X$, there exists a unique $\sigma\in F$ such that $x\in\sigma(\La_F)$.
Note that $\La_F$ is finite since $\cX/F$ is finite.
We let $\cS^+=\{\sigma_1,\ldots,\sigma_d\}$ be a free generator of $F$.
We call the pair $(\La_F,\cS^+)$ a \emph{coordinate} of $X$.
For any $\tau\in F$, we let $(n_1(\tau),\ldots,n_d(\tau))\in\Z^d$ be such that 
$\tau=\sigma_1^{n_1(\tau)}\cdots\sigma_d^{n_d(\tau)}$.  
We define the \emph{block coordinate} of  $x\in X$ for $(\La_F,\cS^+)$
to be $(n_1(\tau),\ldots,n_d(\tau))\in\Z^d$,
where $\tau\in F$ is such that $x\in\tau(\La_F)$.
We let $u_{\cS^+}\colon X\rightarrow\Z^d$
mapping any $x\in X$ to its block coordinate.
For the set of edges
\begin{equation}\label{eq: block}
	\bbE_{\cS}\coloneqq\{(\bar x,\bar x') \in\Z^d\times\Z^d\mid 
	\exists e\in E, u_{\cS^+}(o(e))=\bar x, u_{\cS^+}(t(e))=\bar x'  \},
\end{equation}
the pair  $\Z_{\cS}\coloneqq (\Z^d,\bbE_{\cS})$ form a connected locally finite graph
such that the block coordinate gives a morphism of multi-graphs $u_{\cS^+}\colon\cX\rightarrow\Z_{\cS}$.

\begin{definition}
	\begin{enumerate}
	\item
	We define the \emph{block distance} on $\cX=(X,E)$ to be the function 
	$d_{\cS}\colon X\times X\rightarrow\R$
	given by
	\[
		d_{\cS}(x,x')\coloneqq d_{\Z_{\cS}}(u_{\cS^+}(x),u_{\cS^+}(x'))
	\]
	for any $x,x'\in X$, where $d_{\Z_{\cS}}$ denotes the graph distance on $\Z_{\cS}$.
	Note that by definition,  we have
	 \begin{equation*}
	  	d_\cX(x,x')\leq\diam(\La_F)d_{\cS}(x,x').
	\end{equation*}
	The block distance is not a distance in the usual sense since we may have $d_{\cS}(x,x')=0$
	even if $x\neq x'$.
	\item We define the \emph{block ball} in $\cX=(X,E)$ 
	with center $x\in X$ and radius $R\geq 0$ by
	\[
		\sB(x,R)\coloneqq\{x'\in X\mid d_\cS(x,x') \leq R\},
	\]
	which is a finite connected subset of $X$.
	\item For any $\La\in\sI_X$, we define the \emph{block diameter} of $\La$ by
	\[
		\diam_{\cS}(\La)\coloneqq\inf_{x,x'\in\La}d_{\cS}(x,x').
	\]
	\end{enumerate}
\end{definition}

\begin{definition}\label{def: EE}
	We say that a crystal lattice $\cX=(X,E)$ is \emph{essentially Euclidean} for the action of $G$,
	if there exists a free subgroup $F\subset G$ of finite index with
	coordinate $(\La_F,\cS^+)$ such that for the graph $\Z_{\cS}=(\Z^d,\bbE_{\cS})$, we have
	\begin{equation}\label{eq: EEC}
		d_{\Z_{\cS}}=d_{\Z^d}
	\end{equation}
	on $\Z^d$, where $d_{\Z^d}$ denotes the graph distance on the Euclidean lattice  $\Z^d=(\Z^d,\bbE^d)$.
	In this case, $\Z_\cS$ coincides with $\Z^d$ except for possibly
	simple loops in $\Z_\cS$.
\end{definition}

If $\cX$ is a crystal lattice which is essentially Euclidean for the action of $G$,
and if we fix a coordinate $(\La_F,\cS^+)$ satisfying \eqref{eq: EEC},
then by definition, the topology of $\Z_{\cS}$ corresponds to that of the 
Euclidean lattice $\Z^d=(\Z^d,\bbE^d)$.
Any subset $\bar\La\subset\Z^d$ is connected
if and only if the inverse image
\[
	\La\coloneqq u^{-1}_{\cS^+}(\bar\La)
\]
is connected in $\cX$.  

\begin{lemma}
	Suppose $\cX$ is essentially Euclidean for the action of $G$,
	and fix a coordinate  $(\La_F,\cS^+)$ satisfying \eqref{eq: EEC}.
	For any $x\in X$ and $R\geq 0$, the complement $X\setminus\sB(x,R)$
	is a connected infinite multi-graph if $\dim\cX>1$, and has exactly two connected components
	which are both infinite multi-graphs if $\dim\cX=1$.
\end{lemma}

\begin{proof}
	This follows from the assertion for the Euclidean lattice $\Z^d$,
	noting that the connected components of  $X\setminus\sB(x,R)$
	corresponds bijectively via $u_{\cS^+}$
	with the connected components of $\Z^d\setminus 
	\{ \tilde x' \in \Z^d \mid d_{\Z^d}(\tilde x',u_{\cS^+}(x))\leq R\}$.
\end{proof}

For the rest of this section, we assume that $\cX=(X,E)$ is a crystal lattice which is essentially
Euclidean.  Let $\cM\coloneqq\bsxi_{\!X}(S^X_*)$, where
\begin{equation}\label{eq: bsxi}
	\bsxi_{\!X}\colon S^X_*\rightarrow\Hom_\R(\C{S},\R), \qquad
	\e\mapsto (\xi \mapsto\xi_X(\e))
\end{equation}
for $\xi_X\coloneqq\sum_{x\in X}\xi_x\in C^0_\unif(S^X)$.
We will construct a certain pairing $h_f\colon\cM\times\cM\rightarrow\R$
for functions $f\in C(S^X_*)$ whose differential
is uniform.  This pairing will used in \S \ref{subsec: UL}
to give a criterion for $f$ to be a uniform function.
For an integer $r\in\N$, let
\[
	\sD^r\coloneqq\{ \sB(x,r') \mid x\in X, r'\in\N,\abs{\sB(x,r')}\geq r\}
\]
and
\[
	\sB^r_R\coloneqq\{ (D_1,D_2) \mid D_1,D_2\in\sD^r, d_{\cS}(D_1,D_2)>R\},
\]
where $d_{\cS}(D_1,D_2)\coloneqq\min_{x\in D_1,x'\in D_2}d_\cS(x,x')$.
We define a relation in $\sB^r_R$ by
$(D_1,D_2)\leftrightarrow(D_1,D_2')$ if 
$D_2$ and $D_2'$ are in the same connected component of
$X\setminus\sB(D_1,R)$, and similarly 
$(D_1,D_2)\leftrightarrow(D_1',D_2)$ if 
$D_1$ and $D_1'$ are in the same connected component of
$X\setminus\sB(D_2,R)$.  

For any $\La\subset X$ and $R\geq 0$, 
we let $\sB(\La,R)\coloneqq\bigcup_{x\in X}\sB(x,R)$.

\begin{lemma}\label{lem: independence}
	Let $f\in C(S^X_*)$ such that $\partial_{\cX}f\in C^1_R(S^X)$ for some $R>0$.
	For any finite connected $\La_1,\La_2,\La_2'$, suppose $\La_2$ and $\La_2'$ are in the same 
	connected component of $X\setminus\sB(\La_1,R)$.
	Then the functions $h^{\La_1,\La_2}$ and $h^{\La_1,\La_2'}$ of \cref{prop: pairing} satisfy
	\[
		h^{\La_1,\La_2}_f(\alpha,\beta)=h^{\La_1,\La_2'}_f(\alpha,\beta)
	\]
	for any $\alpha\in\cM_{\abs{\La_1}}$ and $\beta\in\cM_{\abs{\La_2}}\cap\cM_{\abs{\La_2'}}
	=\cM_{\min\{\abs{\La_2},\abs{\La_2'}\}}$.
\end{lemma}

\begin{proof}
	Let  $\alpha\in\cM_{\abs{\La_1}}$ and $\beta\in\cM_{\abs{\La_2}}\cap\cM_{\abs{\La_2'}}
	=\cM_{\min\{\abs{\La_2},\abs{\La_2'}\}}$.
	Since $\La_1\cap \La_2=\La_1\cap \La_2'=\emptyset$, there exists configurations 
	$\e,\e'\in S^X_*$ such that 
	\begin{align*}
		\e&=\e|_{\La_1\cup \La_2},&
		\e'&=\e'|_{\La_1\cup \La_2'},&
	     \e|_{\La_1}&=\e'|_{\La_1},&
		\bsxi_{\!\La_1}(\e)&=\bsxi_{\!\La_1}(\e')=\alpha,&
		\bsxi_{\!\La_2}(\e)&=\bsxi_{\!\La_2'}(\e')=\beta.
	\end{align*}
	Let $Y\subset X\setminus\sB(\La_1,R)$ be a finite connected subset such that $\La_2\cup \La_2'\subset Y$.
	By construction, $\bsxi_Y(\e)=\bsxi_Y(\e')=\beta$ and $\e|_{X\setminus Y}=\e'|_{X\setminus Y}$.
	If we let $W=B(Y,R)$, then by \cref{lem: dependence}, we have
	\[
		(f-\iota^W f)(\e)=(f-\iota^W f)(\e').
	\]
	Since $W\cap(\La_1\cup \La_2)=\La_2$ and $W\cap(\La_1\cup \La_2')=\La_2'$,
	we have
	\begin{align*}
		(f-\iota^W f)(\e)&=(f-\iota^W f)(\e|_{\La_1\cup \La_2})=\iota^{\La_1\cup \La_2}f(\e)-\iota^{\La_2}f(\e)\\
		(f-\iota^W f)(\e')&=(f-\iota^W f)(\e'|_{\La_1\cup \La_2'})=\iota^{\La_1\cup \La_2'}f(\e')-\iota^{\La_2}f(\e').
	\end{align*}
	Combined with the fact that $\iota^{\La_1} f(\e)=\iota^{\La_1}f(\e')$,
	we see that
	\[
		h_f^{\La_1,\La_2}(\alpha,\beta)=h_f^{\La_1,\La_2}(\bsxi_{\!\La_1}(\e),\bsxi_{\!\La_2}(\e))
		=h_f^{\La_1,\La_2'}(\bsxi_{\!\La_1}(\e'),\bsxi_{\!\La_2'}(\e'))=h_f^{\La_1,\La_2'}(\alpha,\beta)
	\]
	as desired.
\end{proof}	

\begin{definition}\label{def: pairing}
	Let $f\in C(S^X_*)$ such that $\partial_{\cX}f\in C^1_R(S^X)$ for some $R>0$.
	We define the paring $h_f\colon\cM\times\cM\rightarrow\R$ on $\cM=\bigcup_{k\in\N}\cM_k$
	as follows.
	We let $k\in\N$ such that $\alpha,\beta\in\cM_k$.  By taking an arbitrary $(D_1,D_2)\in\sB^k_R$,
	we let
	\[
		h_f(\alpha,\beta)\coloneqq h^{D_1,D_2}_f(\alpha,\beta).
	\]
\end{definition}

If $(D_1,D_2),(D_1',D_2')\in\sB^k_R$, if we take $D''_2\in\sD^k$ such that 
$D_2''\subset X\setminus(\sB(D_1,R)\cup\sB(D_1',R))$, then we have
\[
	( D_1,D_2)\leftrightarrow (D_1, D_2'')\leftrightarrow (D_1', D_2'')\leftrightarrow (D_1',D_2').
\]
Using \cref{lem: independence} and \eqref{eq: sym}, we may prove that
\[
	h^{D_1,D_2}_f(\alpha,\beta)=h^{D_1, D_2''}_f(\alpha,\beta)=
	h^{D_1', D_2''}_f(\alpha,\beta)=h^{D_1',D_2'}_f(\alpha,\beta), 
\]
hence that the pairing $h_f$ is independent of the choice of $(D_1,D_2)\in\sB^k_R$.

\begin{proposition}\label{prop: cocycle}
	Let $f\in C(S^X_*)$ such that $\partial_{\cX}f\in C^1_R(S^X)$ for some $R>0$.
	Then the pairing $h_f\colon\cM\times\cM\rightarrow\R$ satisfies the cocycle condition
	\begin{equation}\label{eq: cocycle}
		h_f(\alpha,\beta)+h_f(\alpha+\beta,\gamma)=h_f(\beta,\gamma)+h_f(\alpha,\beta+\gamma)
	\end{equation}
	for any $\alpha,\beta,\gamma\in\cM$.
\end{proposition}

\begin{proof}
	Take $k>0$ sufficiently large so that $\alpha,\beta,\gamma\in\cM_k$.
	We take $ D_1, D_2, D_3\in\sD^k$ sufficiently apart so that 
	$( D_1, D_2), ( D_1, D_3), ( D_2, D_3)\in\sB^k_R$ and
	$( D_1, D_2)\leftrightarrow( D_1, D_3)\leftrightarrow( D_2, D_3)$.
	We may take $\bar D_1,\bar D_2,\bar D_3$ to be balls of radius $k$ in the Euclidean lattice with
	center $(-2R-2k-2, 0,\ldots,0)$, $(0,\ldots,0)$,
	$(2R+2k+2,0,\ldots,0)\in\Z^d$, and we let $D_i\coloneqq u_{\cS^+}^{-1}(\bar D_i)$ for $i=1,2,3$.
	We let $\e\in S^X_*$ be such that $\e=\e|_{ D_1\cup D_2\cup D_2}$
	and $\bsxi_{\!D_1}(\e)=\alpha$, $\bsxi_{\!D_2}(\e)=\beta$, 
	$\bsxi_{\! D_3}(\e)=\gamma$.
	Next, take any $ D_{12}, D_{23}\in\sD^k$ such that 
	\begin{align*}
		 D_1\cup D_2&\subset D_{12}\subset X\setminus\sB( D_3,R),&
		 D_2\cup D_3&\subset D_{23}\subset X\setminus\sB( D_1,R).
	\end{align*}	
	We may take for example $D_{12}\coloneqq u_{\cS^+}^{-1}(\bar D_{12})$ 
	and $D_{23}\coloneqq u_{\cS^+}^{-1}(\bar D_{23})$, where
	$\bar D_{12}$ and $\bar D_{23}$ are balls of radius $R+2k+1$
	with centers $(-R-k-1,0,\ldots,0)$, $(R+k+1,0,\ldots,0)$ in $\Z^d$.
	Then $\bsxi_{\! D_{12}}(\e)=\alpha+\beta$ and $\bsxi_{\! D_{23}}(\e)=\beta+\gamma$.
	By \cref{lem: independence} and the definition of $h_f$, we have
	\begin{align*}
		h_f(\alpha,\beta)&=h_f^{ D_1, D_2}(\alpha,\beta)=\iota^{ D_1\cup   D_2}f(\e)
		-\iota^{ D_1}f(\e)-\iota^{ D_2}f(\e)\\
		h_f(\alpha+\beta,\gamma)&=h_f^{ D_{12}, D_3}(\alpha+\beta,\gamma)
		=\iota^{ D_{12}\cup   D_3}f(\e)
		-\iota^{ D_{12}}f(\e)-\iota^{ D_3}f(\e)\\
		h_f(\beta,\gamma)&=h_f^{ D_2, D_3}(\alpha,\beta)=\iota^{ D_2\cup   D_3}f(\e)
		-\iota^{ D_2}f(\e)-\iota^{ D_3}f(\e)\\
		h_f(\alpha,\beta+\gamma)&=h_f^{ D_1, D_{23}}(\alpha,\beta+\gamma)=\iota^{ D_1\cup D_{23}}f(\e)
		-\iota^{ D_1}f(\e)-\iota^{ D_{23}}f(\e).
	\end{align*}
	Our assertion follows from the fact that
	\begin{align*}
		\iota^{ D_{12}\cup D_{3}}f(\e)&=\iota^{ D_1\cup D_{23}}f(\e)=\iota^{ D_1\cup D_2\cup D_3}f(\e)
	\end{align*}
	and $\iota^{ D_{12}}f(\e)=\iota^{ D_1\cup D_2}f(\e)$, $\iota^{ D_{23}}f(\e)=\iota^{ D_2\cup D_3}f(\e)$.
\end{proof}

%
\subsection{Criterion for Uniform Locality}\label{subsec: UL}
%

In this subsection, we prove a criterion for uniformity.
Let $\cX=(X,E)$ be a crystal lattice for the action of a group $G$ which is essentially Euclidean,
and we fix a coordinate $(\La_F,\cS^+)$ satisfying \eqref{eq: EEC}.
Assume that $(S,\phi)$ is an interaction which is irreducibly quantified.

\begin{proposition}\label{prop: criterion}
	Let $f\in C(S^X_*)$ such that $\partial_{\cX}f\in C^1_R(S^X)$ for some $R>0$.
	If the pairing $h_f\colon\cM\times\cM\rightarrow\R$ satisfies $h_f\equiv0$,
	then $f$ is uniform.
\end{proposition}

We will give a proof of \cref{prop: criterion} at the end of this section.
We first give a condition for a function to be uniform.

\begin{lemma}\label{lem: criterion}
	Suppose $f\in C(S^X_*)$.
	If there exists integer $R>0$ such that for any finite $\La\subset X$ and $x\in X$,
	we have
	\begin{equation}\label{eq: sum}
		\iota^\La f-\iota^{\La\setminus\sB(x,0)}f=\iota^{\La\cap\sB(x,R)} f-\iota^{\La\cap\sB^*(x,R)}f,
	\end{equation}
	then $f$ is a uniform function in $C_\unif(S^X)$.
	Here, $\sB^*(x,R)\coloneqq\sB(x,R)\setminus\sB(x,0)$.
\end{lemma}

\begin{proof}
	First observe that for any $f\in C(S^X_*)$, we have
	\begin{align*}
		\iota^\La f&=\sum_{\La''\subset\La}f_{\La''}
		=\sum_{\La''\subset\La\setminus\sB(x,0)}f_{\La''}
		+\sum_{\substack{\La''\subset\La,\\\La''\cap\sB(x,0)\neq\emptyset}}f_{\La''}\\
		&=\iota^{\La\setminus\sB(x,0)}f +\sum_{\substack{\La''\subset\La,\\\La''\cap\sB(x,0)\neq\emptyset}}f_{\La''}.
	\end{align*}
	If \eqref{eq: sum} holds, then we have
	\begin{equation}\label{eq: sum2}
		\sum_{\substack{\La''\subset\La,\\ \La''\cap\sB(x,0)\neq\emptyset}}f_{\La''}
		=\sum_{\substack{\La''\subset\La\cap\sB(x,R),\\ \La''\cap\sB(x,0)\neq\emptyset}}f_{\La''}.
	\end{equation}
	Let $\sI'\coloneqq\{ \La\in\sI_X\mid f_{\La}\neq 0, \diam_{\cS}(\La)>R \}$,
	and assume that $\sI'\neq\emptyset$.
	We take $\La$ to be a minimal element in $\sI'$ with respect to inclusions, so that
	if $\La''\subsetneq\La$, then we have $f_{\La''}=0$ or $\diam_{\cS}(\La'')\leq R$.
	For such $\La$, take $x\in\La$ such that there exists $y\in\La$ with $d_\cS(x,y)>R$.
	Then we have
	\[
		\sum_{\substack{\La''\subset\La,\\ \La''\cap\sB(x,0)\neq\emptyset}}f_{\La''}
		=f_\La + \sum_{\substack{\La''\subsetneq\La,\\ \La''\cap\sB(x,0)\neq\emptyset}}f_{\La''}
		=f_\La+\sum_{\substack{\La''\subset\La\cap\sB(x,R),\\ \La''\cap\sB(x,0)\neq\emptyset}}f_{\La''}.
	\]
	Comparing this with equality \eqref{eq: sum2}, we see that $f_\La\equiv0$.
	This contradicts the fact that $\La\in\sI'$, hence that $\sI'=\emptyset$.
	In other words, any $\La\in\sI_X$ such that $f_\La\neq0$ satisfies $\diam_{\cS}(\La)\leq R$.
	If we let $R'\coloneqq \diam(\La_0)R$,
	then $\diam(\La)\leq\diam(\La_0)\diam_{\cS}(\La)\leq R'$.
	We see that
	\[
		f=\sum_{\La\in\sI_{\cX,R'}}f_\La,
	\]
	which proves that $f$ is uniform as desired.
\end{proof}

\begin{proposition}\label{prop: wow}
	Let $f\in C(S^X_*)$ such that $\partial_{\cX}f\in C^1_R(S^X)$ for some $R>0$.
	If the pairing $h_f\colon\cM\times\cM\rightarrow\R$ satisfies $h_f\equiv0$,
	then for any $x\in X$ and $\La\in\sI_X$ such that $d_\cS(x,\La)>R$, 
	we have
	\[
		\iota^{\La\cup\sB(x,R)}f-\iota^{\La\cup\sB^*(x,R)}f
		=\iota^{\sB(x,R)}f-\iota^{\sB^*(x,R)}f.
	\]
\end{proposition}

\begin{proof}
	Suppose $\e=\e|_{\La\cup\sB(x,R)}$.
	If we let $Y\coloneqq X\setminus\sB(x,R)$, then $Y$ is connected if $\dim\cX>1$, 
	and $Y=Y_1\cup Y_2$ for two infinite connected components $Y_1,Y_2$ if $\dim\cX=1$.
	Let $\La_i=Y_i\cap \La$ for $i=1,2$ so that $\La=\La_1\cup\La_2$, with the convention $Y_2=\La_2=\emptyset$
	if $\dim\cX>1$.  
	For $\sD\coloneqq\bigcup_{k\in\N}\sD^k$,
	let $D_1'\in\sD$ be such that $\sB(\La,R)\cup\sB(x,R)\subset D_1'$, and
	let $D_1\in\sD$ such that 
	\begin{align*}
		d_\cS(\La,D_1)&>R,& 
		D_1&\subset(X\setminus\sB(D'_1,R))\cap Y_1,&
		\abs{D_1}&\geq\abs{\La_1}.  
	\end{align*}	
	Such $D_1$ exists since $Y_1$ is an infinite set.
	Choose $\e'\in S^X_*$ such that $\e'$ coincides with $\e$ outside $\La_1\cup D_1\subset Y_1$,
	is at base state on $\La_1$, and $\bsxi_{\!D_1}(\e')=\bsxi_{\!\La_1}(\e)$.
	This implies that $\bsxi_{\!Y_1}(\e')=\bsxi_{\!Y_1}(\e)$.
	Hence by \cref{lem: dependence} applied to $Y_1$ and $W=X\setminus\{x\}$,
	we have
	\[
		f(\e)-f(\e|_W)=f(\e')-f(\e'|_W).
	\]
	Since $\e=\e|_{D_1'}$ and $\e'=\e'|_{D_1\cup D_1'}$, we see that
	\begin{equation}\label{eq: result}
		f(\e|_{D_1'})-f(\e|_{D_1'\setminus\sB(x,0)})=f(\e'|_{D_1\cup D_1'})-f(\e'|_{D_1\cup (D_1'\setminus\sB(x,0)}).
	\end{equation}
	By construction, $d_\cS(D_1,D'_1)>R$, hence the calculation of $h_f$ via the pairs $(D_1,D'_1)\in\sB_R$ 
	and $(D_1,D_1'\setminus\sB(x,0))$ as well as our assumption that $h_f\equiv 0$ gives
	\begin{align*}
		f(\e'|_{D_1\cup D_1'})&=f(\e'|_{D_1})+f(\e'|_{D_1'}),&
		f(\e'|_{D_1\cup(D_1'\setminus\sB(x,0))})&=f(\e'|_{D_1})+f(\e'|_{D_1'\setminus\sB(x,0)}).
	\end{align*}
	Hence \eqref{eq: result} gives
	\[
		f(\e|_{D_1'})-f(\e|_{D_1'\setminus\sB(x,0)})=f(\e'|_{D_1'})-f(\e'|_{D_1'\setminus\sB(x,0)}).
	\]
	In particular, since  $\e=\e|_{\La\cup\sB(x,R)}$ and $\e'=\e'|_{D_1\cup\La_2\cup\sB(x,R)}$,
	we have
	\[
		f(\e|_{\La\cup\sB(x,R)})-f(\e|_{\La\cup\sB^*(x,R)})=
		f(\e'|_{\La_2\cup\sB(x,R)})-f(\e'|_{\La_2\cup\sB^*(x,R)}).
	\]
	Since $\e'|_{X\setminus Y_1}=\e|_{X\setminus Y_1}$, this gives
	\[
		f(\e|_{\La\cup\sB(x,R)})-f(\e|_{\La\cup\sB^*(x,R)})
		=f(\e|_{\La_2\cup\sB(x,R)})-f(\e|_{\La_2\cup\sB^*(x,R)}).
	\]
	This gives our assertion for $\dim\cX>1$, since $\La_2=\emptyset$ in this case.
	The proof for the case $\dim\cX=1$ follows by applying the same argument as that of $\La_1$
	to $\La_2$.
\end{proof}

We may now prove \cref{prop: criterion}.

\begin{proof}[Proof of \cref{prop: criterion}]
	Suppose $f\in C(S^X_*)$ such that $\partial_{\cX}f\in C^1_R(S^X)$ for some $R>0$,
	and assume that $h_f\equiv0$ for the paring of \cref{def: pairing}.
	By \cref{lem: criterion}, it is sufficient to prove that for any $\La\in\sI_X$ and $x\in\La$,
	we have
	\[
		\iota^\La f-\iota^{\La\setminus\sB(x,0)}f=\iota^{\La\cap\sB(x,R)} f-\iota^{\La\cap\sB^*(x,R)}f.
	\]
	Let $\La'\coloneqq\La\setminus(\La\cap\sB(x,R))$. 
	By \cref{prop: wow} applied to $\La'$, we have
	\begin{equation}\label{eq: apply}
		\iota^{\La'\cup\sB(x,R)} f-\iota^{\La'\cup\sB^*(x,0)}f
		=\iota^{\sB(x,R)} f-\iota^{\sB^*(x,R)}f.
	\end{equation}
	By applying $\iota^{\La'\cup(\La\cap\sB(x,R))}$ to both sides of \eqref{eq: apply},
	we have
	\[
		\iota^{\La'\cup(\La\cap\sB(x,R))} f-\iota^{\La'\cup(\La\cap\sB^*(x,0))}f
		=\iota^{\La\cap\sB(x,R)} f-\iota^{\La\cap\sB^*(x,R)}f.
	\]
	Since $\La=\La'\cup(\La\cap\sB(x,R))$,
	this coincides with the desired equality.
\end{proof}

%
\subsection{Examples of Essentially Euclidean Crystal Lattices}\label{subsec: EE}
%

In this subsection, we give examples of essentially Euclidean crystal lattices.
In particular, we will give a criterion for a crystal lattice to be essentially Euclidean,
and show that the maximal abelian covering of a crystal is an essentially Euclidean lattice.
The results of this subsection are not necessary for the proof of our main theorem.

In what follows, for any $\La\in\sI_X$, we let $E_\La\coloneqq\{e\in E\mid o(e),t(e)\in\La\}$ and
$\partial E_{\La}\coloneqq\{ e\in E\mid o(e)\in\La,t(e)\not\in\La\}$.

\begin{lemma}\label{lem: maximal}
	Suppose $\cX=(X,E)$ is a crystal lattice for the action of a finitely generated free group $G$.
	Let $d\coloneqq\dim\cX$ be the rank of $G$.  For $\cX/G=(X_0,E_0)$, if
	\begin{equation}\label{eq: x0}
		d=1-\abs{X_0}+\frac{1}{2}\abs{E_0},
	\end{equation}
	then $\cX$ is essentially Euclidean.
\end{lemma}

\begin{proof}
	Let $\La_0$ be a fundamental domain of $\cX$ for the action of $G$.
	In order for the graph $\La_0$ consisting of $\abs{\La_0}$ vertices
	to be connected in the sense of a symmetric
	graphs, one needs at least $\abs{\La_0}-1$ edges to connect
	distinct vertices.  Considering symmetry, this would require $2(\abs{\La_0}-1)$
	edges.
	This implies that $\abs{E_{\La_0}}\geq 2(|\La_0|-1)$, 
	which shows that 
	\[
		1-\abs{\La_0}+\frac{1}{2}\abs{E_{\La_0}}\geq 0.
	\]
	Moreover, we have $\abs{X_0}=\abs{\La_0}$ and $\abs{E_0}=\abs{E_{\La_0}}+\abs{\partial E_{\La_0}}$,
	hence we see from \eqref{eq: x0} that
	\[
		\abs{\partial E_{\La_0}}=2d-2\Bigl(1-\abs{X_0}+\frac{1}{2}\abs{E_{\La_0}}\Bigr)\leq 2d.
	\]
	Let $\cS=\{\sigma\in G\mid \exists e\in \partial E_{\La_0}, t(e)\in\sigma(\La_0)\}$.	
	We show that $\cS$ generates $G$.
	Fix $x_0\in\La_0$.  For any $\tau\in G$, since $X$ is connected, there exists a path
	$\vec\gamma=(e^1,\ldots,e^N)$ from $x_0$ to $\tau(x_0)$.  We let $\tau_0=\id$,
	$\tau_i\in G$ be such that $t(e^i)\in\tau_i(\La_0)$ for $i=1,\ldots,N$,
	and $\sigma_{e^i}\in G$ be such that $t(e^i)=\sigma_{e^i}\tau_{i-1}(\La_0)$.
	Then we have $\sigma_{e^i}=\id$ or $\sigma_{e^i}\in\cS$.
	Then
	\[
		\tau_N=\sigma_{e^N}\tau_{N-1}=\sigma_{e^N}\sigma_{e^{N-1}}\tau_{N-2}=\cdots=\sigma_{e^N}\cdots\sigma_{e^1},
	\]
	which shows that $\tau=\tau_N$ is in the subgroup generated by $\cS$.
	Note that if $\sigma\in\cS$, then $\sigma^{-1}\in\cS$. This shows that $\abs{\cS}\geq 2d$
	since $d$ is the rank of $G$.

	For any $e\in\partial E_{\La_0}$, let $\sigma_e\in\cS$ be such that $t(e)\in\sigma_e(\La_0)$.
	Then $e\mapsto\sigma_e$ gives a map $\partial E_{\La_0}\rightarrow\cS$
	which is surjective from the definition of $\cS$.
	Since $\abs{\partial E_{\La_0}}\leq 2d$, we must have $\abs{\cS}=\abs{\partial E_{\La_0}}=2d$.
	If we fix any $\cS^+\subset\cS$ such that either $\sigma\in\cS^+$ or $\sigma^{-1}\in\cS^+$
	for any $\sigma\in\cS$, then $\cS^+$ is a free generator of $G$, and 
	$(\La_0,\cS^+)$ gives a coordinate of $\cX$ which satisfies \eqref{eq: EEC}.
\end{proof}

\begin{lemma}\label{lem: maximal2}
	Let $\cX_0=(X_0,E_0)$ be a connected crystal, in other words,
	a strictly symmetric finite connected multi-graph.
	Then the maximal abelian covering $\cX^\ab$ of $\cX_0$ 
	has a free action of the finitely generated abelian group
	$G=\pi_1^\ab(\cX_0)$, where 
	$\pi_1^\ab(\cX_0)$
	is the maximal abelian quotient of the fundamental group $\pi_1(\cX_0)$
	of $\cX_0$.
	Then $\cX^\ab$ for the action of $G$
	satisfies the conditions of \cref{lem: maximal},
	hence is a crystal lattice which is essentially Euclidean.
\end{lemma}

\begin{proof}
	Since $\cX_0$ is connected,
	the homology group $H_1(\cX_0,\Z)$ of the crystal $\cX_0$ is a free abelian group
	of rank
	\[
		d= 1-\abs{X_0}+\frac{1}{2}\abs{E_0}.
	\]
	The standard isomorphism $\pi^\ab_1(\cX_0)\cong H_1(\cX_0,\Z)$ shows that $\pi^\ab_1(\cX)$
	is a free abelian group of rank $d$.  Hence $\cX^\ab=(X^\ab,E^\ab)$ is a connected locally finite multi-graph
	with a free action of the finitely generated abelian group $G\coloneqq\pi_1^\ab(\cX_0)$, such that
	$\cX^\ab/G=\cX_0$.  This shows in particular that $\cX^\ab$ satisfies the conditions of \cref{lem: maximal}.
\end{proof}

\begin{example}\label{example: more}
	\begin{enumerate}
		\item The \emph{hexagonal lattice} 
		$\cX$ gives an example of a crystal lattice which is essentially Euclidean.
		It has a free action of $G=\Z^2$ given by translation, 
		and the quotient $\cX/G=(X_0,E_0,o,t,\iota)$ is given by $X_0=\{x_0,x_1\}$,
		$E_0=\{e_0,\bar e_0,e_1,\bar e_1,e_2,\bar e_2\}$, $o(e_j)=x_0$, 
		$t(e_j)=x_1$, and $\iota(e_j)=\bar e_j$,  $\iota(\bar e_j)=e_j$ for $j=1,2,3$.
		In fact, $\cX$ is the maximal abelian covering of $\cX/G$.
		\item The \emph{triangular lattice} $\cX$ gives an example of a crystal lattice 
		which is \emph{not} essentially Euclidean. $\cX$ has an action of $G=\Z^2$ by
		translation.  For any coordinate $(\La_F,\cS^+)$, the graph $\Z_\cS=(\Z^2,\bbE_\cS)$
		contains extra diagonal edges, hence the graph distance $d_\cS$ does not coincide with 
		$d_{\Z^2}$.
\begin{figure}[h]
	\centering
	\includegraphics[width=15cm]{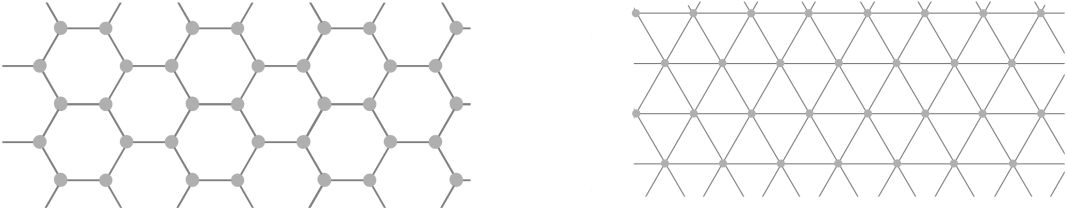}
	\caption{The Hexagonal and Triangular Lattices}
\end{figure}
		\item For integers $d\geq 1$ and $n>0$, 
		the \textit{Euclidean lattice with nearest $n$-neighbor} 
	$\bbZ^d_n=(\bbZ^d,\bbE_n)$, where
	\[
		\bbE_n\coloneqq \bigl\{  (x,y) \in \bbZ^d\times\bbZ^d \,\,\big|\,\,  0<|x-y|\leq n\bigr\},
	\]
	is a crystal lattice for the action of the group $G=\Z^d$ by translation.
	\begin{figure}[h]
			\centering
			\includegraphics[width=14cm]{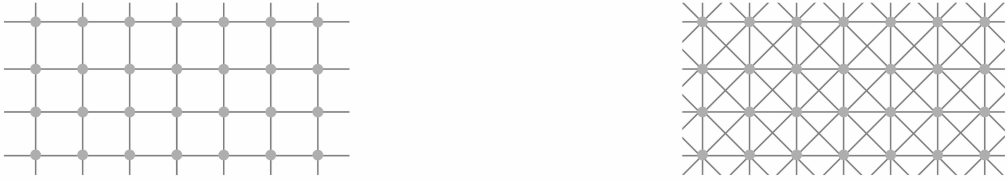}
			\caption{The Euclidean Lattices $\bbZ^2$ and with nearest $2$-neighbor $\bbZ^2_2$}	
			\label{fig: 6}
	\end{figure}
	The lattice $\bbZ^d_n$ for $d>1$ and $n>1$ give examples of crystal lattices which
	are \emph{not} essentially Euclidean.
	\end{enumerate}
\end{example}

%
%
%
\section{Decomposition for Closed Uniform Forms}\label{sec: main}
%
%

%

In this section, we prove our main result, \cref{thm: main},
which gives the decomposition of shift-invariant 
closed uniform forms on configuration spaces with transition structure
over general crystal lattices.  \cref{thm: main} is obtained from \cref{thm: 1}
concerning the surjectivity of the differential from the space of uniform functions
to closed uniform forms.  \cref{thm: 1} is proved by reducing to the essentially Euclidean case,
\cref{thm: 0}.

%
\subsection{The Main Theorem}\label{subsec: main}
%

In this subsection, we will state our main theorem.
Let $\cX=(X,E)$ be a crystal lattice of dimension $\dim\cX$
for the action of an abelian group $G$.
We let $(S,\phi)$ be an interaction which is irreducibly quantified.
The action of $G$ on $\cX$ induces actions on the
spaces $C^0_\unif(S^X)$, $C^1_\unif(S^X)$, $Z^1_\unif(S^X)$ 
compatible with the differential \eqref{eq: differential}.
We denote by  $C^0_\unif(S^X)^G$, $C^1_\unif(S^X)^G$, $Z^1_\unif(S^X)^G$
the $G$-invariant subspaces of $C^0_\unif(S^X)$, $C^1_\unif(S^X)$, $Z^1_\unif(S^X)$
consisting of functions or forms invariant with respect to the action of $G$.

Again let  $d=\dim\cX$.  We fix a generator $\sigma_1,\ldots,\sigma_d$ of the 
free part $F\subset G$  which gives an isomorphism $F\cong\Z^d$.
We may extend this isomorphism to a surjective homomorphism $G\rightarrow\Z^d$
by mapping the torsion subgroup of $G$ to \emph{zero}.
For any $\tau\in G$, we let $(n_1(\tau),\ldots,n_d(\tau))\in\Z^d$ be the image of $\tau$
with respect to the above surjection.
We fix a fundamental domain $\La_0$ of $X$ for the action of $G$.

\begin{lemma}\label{lem: tau}
	For any conserved quantity $\xi\in\C{S}$ and $j=1,\ldots,d$, let
	\[
		\frA^j_\xi\coloneqq\sum_{\tau\in G}n_j(\tau)\xi_{\tau(\La_0)}\in C^0_\unif(S^X).
	\]
	Then for any $j,k=1,\ldots,d$, the functions $\frA^j_\xi$ satisfies 
	\[
		(1-\sigma_k)\frA^j_\xi=
		\begin{cases}
			\xi_X  &j=k,\\
			0 & j\neq k.
		\end{cases}
	\]
	In particular, we have $\partial_{\cX}\frA^j_\xi \in Z^1_\unif(S^{X})^G$.
\end{lemma}

\begin{proof}
	By \cref{def: UL}, $\frA^j_\xi=\sum_{\tau\in G}
	n_j(\tau)\sum_{x\in\tau(\La_0)}\xi_x$ is a uniform function,
	hence
	\[
		\partial_{\cX}\frA^j_\xi\in Z^1_\unif(S^X).
	\] 
	For any $\sigma\in G$, we have
	\begin{align*}
		(1-\sigma)\frA^j_\xi& =\sum_{\tau\in G}n_j(\tau)\xi_{\tau(\La_0)}-\sum_{\tau\in G}n_j(\tau)\xi_{\sigma\tau(\La_0)}
		=\sum_{\tau\in G}n_j(\tau)\xi_{\tau(\La_0)}-\sum_{\tau\in G}n_j(\sigma^{-1}\tau)\xi_{\tau(\La_0)}\\
		&=\sum_{\tau\in G}n_j(\tau)\xi_{\tau(\La_0)}-\sum_{\tau\in G}(n_j(\tau)-n_j(\sigma))\xi_{\tau(\La_0)}
		=n_j(\sigma)\sum_{\tau\in G}\xi_{\tau(\La_0)}=n_j(\sigma)\xi_X
	\end{align*}
	for $j=1,\ldots, d$.  Since $\partial_{\cX}\xi_X=0$ by \cite{BKS20}*{Lemma 2.20} (see also \cref{thm: 0}), 
	the compatibility of the group action with respect
	to the differential gives $(1-\sigma)\partial_{\cX}\frA^j_\xi=\partial_{\cX}\xi_X=0$, which shows that $\partial_{\cX}\frA^j_\xi \in Z^1_\unif(S^{X})^G$
	as desired.
\end{proof}

Denote by 
\[
	\cV\coloneqq\Span_\R\{  \frA^j_{\xi}\mid \xi\in\C{S}, j=1,\ldots, d \}
\] 
the $\R$-linear subspace
of $C^0_\unif(S^{X})$ spanned by $\frA^j_{\xi}$ for $\xi\in\C{S}$ and $j=1,\ldots, d$.
Note that the space $\cV$ depends on the choice of the fundamental domain $\La_0$.
We let $\cC$ and $\cE$ be the spaces of shift-invariant closed and exact uniform forms
given by 
\begin{align*}
	\cC&\coloneqq Z^1_\unif(S^{X})^G,  &
	\cE&\coloneqq\partial_{\cX}(C^0_\unif(S^{X})^G).
\end{align*}

\begin{lemma}\label{lem: cd}
	Let $\partial_{\cX}\cV$ be the $\R$-linear space 
	\[
		\partial_{\cX}\cV\coloneqq\Span_{\R}\{\partial_{\cX}\frA^j_{\xi}\mid \xi\in\C{S}, j=1,\ldots,d\}.
	\] 
	If we map $(\zeta^{(1)},\ldots,\zeta^{(d)})\in\bigoplus_{j=1}^d\C{S}$
	to $\frA\coloneqq\sum_{j=1}^d\partial_{\cX}\frA^j_{\zeta^{(j)}}$, then we have an $\R$-linear isomorphism
	\[
		\bigoplus_{j=1}^d\C{S}\cong\partial_{\cX}\cV.
	\]
	In particular, we have
	$\dim_\R\partial_{\cX}\cV=c_\phi d$
	for $c_\phi=\dim_\R\C{S}$.
\end{lemma}

\begin{proof}
	The map is $\R$-linear and surjective by the definition of $\frA^j_\xi$ and $\partial_\cX\cV$.
	It is sufficient to prove that the map is injective.
	Suppose 
	we have $\partial_{\cX}\frA=0$ for
	\begin{align*}
		\frA=\sum_{j=1}^d\frA^j_{\zeta^{(j)}},
		\qquad (\zeta^{(1)},\ldots,\zeta^{(d)})\in\bigoplus_{j=1}^d\C{S}.
	\end{align*}
	By \cref{thm: old}, we have $\C{S}\cong\Ker\partial_{\cX}$,
	hence $\frA=\xi_X$ for some $\xi\in\C{S}$.
	Since $\xi_X(\e)=\xi_X(\sigma(\e))$ for any $\sigma\in G$, 
	we see that $\xi_X\in C^0_\unif(S^X)^G$.  This shows that
	\[
		(1-\sigma_j)\frA=\zeta^{(j)}_X=0,\qquad j=1,\ldots,d,
	\]	
	where $\zeta^{(j)}_X$ is the image of $\zeta^{(j)}\in\C{S}$ in $C^0_\unif(S^X)$
	through the isomorphism $\C{S}\cong\Ker\partial_{\cX}$ of \cref{thm: old}.
	This proves that $\zeta^{(j)}=0$ for any $j=1,\ldots,d$, hence
	that our map is injective as desired.
\end{proof}

We say that an interaction $(S,\phi)$ is \textit{simple}, if $c_\phi=1$,
and for any nonzero conserved quantity $\xi\in\Consv^\phi(S)$,
the monoid generated by $\xi(S)$ via addition in $\bbR$ 
is isomorphic to $\bbN$ or $\bbZ$.
Our main theorem is as follows.

\begin{theorem}\label{thm: main}
	Let $\cX=(X,E)$ be a crystal lattice for a finitely generated abelian group $G$.
	Assume that the interaction $(S,\phi)$ is irreducibly quantified, and assume in addition
	that $(S,\phi)$ is simple if $\dim\cX=1$.
	Then we have a decomposition
	\[
		\cC\cong\cE\oplus\partial_{\cX}\cV.
	\]	
\end{theorem}

\begin{remark}
	In \cite{BKS20}*{Theorem 5}, we proved the theorem when the locale $\cX$ is weakly transferable
	or transferable.  Unfortunately, a crystal lattice does not in general satisfy this condition.
	We modify the method of  \cite{BKS20} to prove Theorem \ref{thm: main}.
\end{remark}

%
\subsection{Surjectivity of the Differential}\label{subsec: EEC}
%

In this subsection, we will prove the following Theorem \ref{thm: 1} for a general crystal lattice $\cX=(X,E)$.
This theorem plays a crucial role in the proof of our main theorem.

\begin{theorem}\label{thm: 1}
	Let $\cX=(X,E)$ be a crystal lattice, and let $(S,\phi)$
	be an interaction which is irreducibly quantified.  We assume in addition that $(S,\phi)$
	is simple if $\dim\cX=1$.   Then the differential
	\[
		\partial_{\cX}\colon C_\unif(S^X)\rightarrow Z^1_\unif(S^X)
	\]
	of \eqref{eq: differential} is surjective.
\end{theorem}

We will give the proof of \cref{thm: 1} at the end of this section.
The proof is obtained by reducing to the essentially Euclidean case.
We first prove the essentially Euclidean case.
Assume now that $\cX=(X,E)$ is a crystal lattice which is essentially Euclidean,
and let the notations be as in \S\ref{subsec: EECL} and \S\ref{subsec: UL}.
The following algebraic lemma will be used in the proof.

\begin{lemma}\label{lem: splitting1}
	Let $\cM$ be a commutative monoid satisfying the cancellation property.
	In other words, for any $\alpha,\beta,\gamma\in\cM$, if 
	$\alpha+\beta=\alpha+\gamma$, then we have $\beta=\gamma$.
	Consider a pairing $H\colon \cM\times\cM\rightarrow\bbR$
	satisfying the cocycle condition
	\[
		H(\alpha,\beta)+H(\alpha+\beta,\gamma)
		=H(\beta,\gamma)+H(\alpha,\beta+\gamma).
	\]
	Assume either of the following conditions.
	\begin{enumerate}
		\item $H$ is \emph{symmetric}.  In other words $H(\alpha,\beta)=H(\beta,\alpha)$ for any $\alpha,\beta\in\cM$.
		\item $\cM\cong\bbN$ or $\cM\cong\bbZ$
		viewed as a commutative monoid with respect to addition.
	\end{enumerate}
	Then there exists a map 
	\[
		h\colon\cM\rightarrow\bbR
	\]
	such that $H(\alpha,\beta)=h(\alpha)+h(\beta) - h(\alpha+\beta)$ for any $\alpha,\beta\in\cM$.
\end{lemma}

\begin{proof}
	(i) is exactly  \cite{BKS20}*{Lemma 5.5} and
	(ii) is exactly \cite{BKS20}*{Lemma 5.6}.
	The idea behind the proof is that the conditions implies that the extension of monoids
	\[
		\xymatrix{
			0\ar[r]&\R\ar[r]&\sE\ar[r]& \cM\ar[r]&0
		}
	\]
	defined by the cocycle $H$ is split.
	See also \cite{BKS20}*{Remark 5.7}.
\end{proof}

If $\dim\cX>1$, then we may prove that the pairing $h_f$ of \eqref{eq: pair} is symmetric as follows:

\begin{lemma}\label{lem: symmetric}
	Assume that $\cX$ is essentially Euclidean and $\dim\cX>1$.
	Then for any $f\in C(S^X_*)$ such that $\partial_{\cX}f\in C^1_\unif(S^X)$,
	the pairing
	\[
		h_f\colon\cM\times\cM\rightarrow\R
	\]
	of \cref{def: pairing} is symmetric. 
\end{lemma}

\begin{proof}
	Let $R>0$ be an integer such that $\partial_{\cX}f\in C^1_R(S^X)$.
	For any integer $k>0$, consider $\alpha,\beta\in\cM_k$.
	If we take $D_1,D_2,D_3$ such that $d_{\cS}(D_i,D_j)>R$ for $i\neq j$,
	then since $\dim\cX>1$, we see that $(D_i,D_j)\in\sD^k_R$ for $i\neq j$, and
	\[
		(D_1,D_2)\leftrightarrow(D_1,D_3)\leftrightarrow(D_2,D_3)\leftrightarrow(D_2,D_1).
	\]
	By \cref{lem: independence} and \eqref{eq: sym}, we see that
	\[
		h^{D_1,D_2}_f(\alpha,\beta)
		=h^{D_1,D_3}_f(\alpha,\beta)
		=h^{D_2,D_3}_f(\alpha,\beta)
		=h^{D_3,D_2}_f(\beta,\alpha)
		=h^{D_1,D_2}_f(\beta,\alpha).
	\]
 	The definition of $h_f$ gives
	\[
		h_f(\alpha,\beta)=h^{D_1,D_2}_f(\alpha,\beta)=h^{D_1,D_2}_f(\beta,\alpha)=h_f(\beta,\alpha)
	\]
	as desired.
\end{proof}

We may now prove the following theorem, which is \cref{thm: 1} for essentially Euclidean crystal lattices.

\begin{theorem}\label{thm: 0}
	Let $\cX=(X,E)$ be a crystal lattice which is essentially Euclidean, and let $(S,\phi)$
	be an interaction which is irreducibly quantified.  We assume in addition that $(S,\phi)$
	is simple if $\dim\cX=1$.   Then the differential
	\[
		\partial_{\cX}\colon C_\unif(S^X)\rightarrow Z^1_\unif(S^X)
	\]
	is surjective.
\end{theorem}

\begin{proof}
	Suppose $\omega\in Z_\unif(S^X) \subset Z^1(S^X_*)$.  
	Since $\omega\in Z^1(S^X_*)$, by \cref{lem: exact}, there exists $f\in C(S^X_*)$
	such that $\partial_{\cX} f=\omega$.  Furthermore, since $\omega\in C^1_\unif(S^X)$,
	there exists $R>0$ such that $\partial_{\cX} f=\omega\in C^1_R(S^X)$.
	Let $h_f\colon \cM\times\cM\rightarrow\bbR$ 
	be the pairing of \cref{def: pairing} which by \cref{prop: cocycle}
	satisfying the cocycle  condition \eqref{eq: cocycle}.
	Note that if $\dim\cX>1$, then by  \cref{lem: symmetric}, the pairing $h_f$ is symmetric.
	Hence by Lemma \ref{lem: splitting1}, 
	for both the cases $\dim\cX=1$ and $\dim\cX>1$,
	there exists $h\colon\cM\rightarrow\bbR$ such that 
	\begin{equation}\label{eq: coboundary}
		h_f(\alpha,\beta)=h(\alpha)+h(\beta)-h(\alpha+\beta)
	\end{equation}
	for any $\alpha,\beta\in\cM$.
	We define the function $g\in C(S^X_*)$ by 
	\[
		g\coloneqq f+h\circ \bsxi_{\!X}.
	\] 
	We will prove that $g$ is uniform.
	By Theorem \ref{thm: old}, for any $\xi\in\C{S}$, we have $\partial_{\cX}\xi_X=0$.
	This implies that $\nabe\xi_X=0$, hence
	$\xi_X(\e^e)=\xi_X(\e)$ for any $e\in E$.
	This shows that $h\circ \bsxi_{\!X}(\e^e)= h\circ \bsxi_{\!X}(\e)$ for any $e\in E$,
	hence $\nabe(h\circ \bsxi_{\!X})=0$, which implies that $\partial_{\cX}(h\circ \bsxi_{\!X})=0$.
	This gives $\partial_{\cX} g=\partial_{\cX} f=\omega$.	
	Furthermore, noting that 
	$\iota^{D_1}(h\circ \bsxi_{\!X}(\e))=h\circ \bsxi_{\!D_1}(\e)$ for any $D_1\subset X$, 
	we have
	\begin{align*}
		\iota^{D_1\cup D_2}g(\e)-\iota^{D_1}g(\e)-\iota^{D_2}g(\e)
		&=\iota^{D_1\cup D_2} f(\e)-\iota^{D_1} f(\e)-\iota^{D_2}f(\e) \\
		&\qquad+h\circ\bsxi_{\!D_1\cup D_2}(\e)- h\circ\bsxi_{\!D_1}(\e)-h\circ\bsxi_{\!D_2}(\e)=0
	\end{align*}
	for any $(D_1,D_2)$ such that $d_{\cS}(D_1,D_2)>R$ and $\e\in S^X$, 
	where we have used the coboundary condition \eqref{eq: coboundary} and the fact that
	$\bsxi_{\!D_1\cup D_2}(\e) = \bsxi_{\!D_1}(\e)+\bsxi_{\!D_2}(\e)$.	
	From the definition of the pairing given in \cref{def: pairing}, 
	we see that $h_g\equiv0$. 
	Hence by Proposition \ref{prop: criterion}, 
	we see that $g\in C_\unif(S^X)$.
	Our assertion is now proved by replacing $f$ by $g$.
\end{proof}

We now return to the case of a general crystal lattice $\cX=(X,E)$.
In order to reduce the proof of \cref{thm: 1} to the essentially Euclidean case,
we first prove the existence of a canonical isomorphism
between the spaces of uniform functions for equivalent locales,
as well as between the spaces of closed uniform forms.

\begin{definition}\label{def: equivalent}
	Given two locales $\cX=(X,E)$ and $\cX'=(X',E')$
	we say that $\cX$ and $\cX'$ are \emph{equivalent},
	if $X=X'$ and there exists constants $C,C'\geq1$ such that
	\begin{align}\label{eq: equivalent}
		 d_{\cX}(x,y)& \leq C d_{\cX'}(x,y),&
		\qquad d_{\cX'}(x,y) &\leq  C' d_{\cX}(x,y)
	\end{align}
	for any  $x,y\in X$. 
\end{definition}

We next prove that the spaces of 
uniform functions and uniform closed forms are isomorphic for equivalent locales.

\begin{lemma}\label{lem: equivalent}
 	Let $\cX=(X,E)$ and $\cX'=(X',E')$ be locales
	which are equivalent in the sense of  \cref{def: equivalent}.
	Assume in addition that $(S,\phi)$ is an interaction which is irreducibly quantified.
	Then there exists canonical isomorphisms
	\begin{align*}
		C^0_\unif(S^{X})&\cong C^0_\unif(S^{X'}),&
		Z^1_\unif(S^{X})&\cong Z^1_\unif(S^{X'})
	\end{align*}
	which are compatible with the differentials $\partial_{\cX}$ and $\partial_{X'}$.
	In other words, we have a commutative diagram
	\[
		\xymatrix{
			C^0_\unif(S^{X})\ar[r]^{\partial_{\cX}}\ar[d]^\cong&Z^1_\unif(S^{X})\ar[d]^\cong\\
			C^0_\unif(S^{X'})\ar[r]^{\partial_{\cX'}}&Z^1_\unif(S^{X'}).
		}
	\]
\end{lemma}

\begin{proof}
	Since $\cX$ and $\cX'$ are equivalent,  we have $X=X'$ and
	there exists constants $C,C\geq1$ such that
	$d_{\cX}(x,y)\leq Cd_{\cX'}(x,y)$ and $d_{\cX'}(x,y)\leq C'd_{\cX}(x,y)$ for any $x,y\in X$.
	This gives
	\begin{align*}
		\diam_{\cX}(\La)& \leq C\diam_{\cX'}(\La), &
		\diam_{\cX'}(\La) &\leq C' \diam_{\cX}(\La) 
	\end{align*}
	for any  $\La\in\sI_X$, which shows that $C^0_\unif(S^X)=C^0_\unif(S^{X'})$ by 
	definition of uniform functions.  
	
	Next, let $E''\coloneqq E\cup E'$ be the disjoint union of $E$ and $E'$, and we let $o,t\colon E''\rightarrow X$
	be the map induced from the origin and target maps on $E$ and $E'$.  Then $E''$ has a natural inversion
	$\iota$
	induced from the inversions on $E$ and $E'$, and we see that $\cX''=(X,E'',o,t,\iota)$ is a locale.
	Since we have natural inclusions $E,E'\hookrightarrow E''$,
	we have	
	\[
		d_{\cX''}(x,y)\leq d_{\cX}(x,y), d_{\cX'}(x,y)
	\]
	for any $x,y\in X$.  Moreover, by construction, we have
	\begin{align*}
		d_{\cX}(x,y)&\leq C d_{\cX''}(x,y), &
		d_{\cX'}(x,y)&\leq C' d_{\cX''}(x,y)
	\end{align*}
	for any $x,y\in X$, hence $\cX''$ is equivalent to both $\cX$ and $\cX'$.
	This shows in particular that $C^0_\unif(S^{X''})=C^0_\unif(S^{X})=C^0_\unif(S^{X'})$
	for $X''=X$.
	The natural morphisms $\cX,\cX'\hookrightarrow\cX''$
	of multi-graphs induces projections 
	\begin{align}\label{eq: projection}
		C^1(S^{X''})&\rightarrow C^1(S^{X}), &
		C^1(S^{X''})&\rightarrow C^1(S^{X'})
	\end{align}
	given through \eqref{eq: inclusion}
	by the natural projections $\prod_{e\in E''}C(S^{X''}_*)\rightarrow\prod_{e\in E}C(S^{X}_*)$
	and $\prod_{e\in E''}C(S^{X''}_*)\rightarrow\prod_{e\in E'}C(S^{X'}_*)$
	which are compatible with the differentials. 
	
	By definition of the differential,
	the subspaces $ \Ker\partial_{\cX},  \Ker\partial_{\cX'},\Ker\partial_{\cX''}$
	of $C^0(S^{X}_*)$ coincide with the functions on $S^X_*$ which are constant
	on the connected components of respectively $S^X_*$, $S^{X'}_*$, $S^{X''}_*$.
	If we let
	\[
		\bsxi_{\!X}\colon S^X_*\rightarrow\Hom_\R(\C{S},\R), \qquad
		\e\mapsto (\xi \mapsto\xi_X(\e))
	\]
	be the map given in \eqref{eq: bsxi}, then since we have assumed that the
	interaction $(S,\phi)$ is irreducibly quantified, the connected components 
	of $S^X_*$ corresponds bijectively with the elements in $\cM=\bsxi_X(S^X_*)$
	(see \cite{BKS20}*{Remark 2.33}).  Since the definition of the map $\bsxi_{\!X}$
	depends only on $X$ and is independent of the edges of the locales $\cX$, $\cX'$, $\cX''$,
	we see that the connected components of $S^X_*$, $S^{X'}_*$, $S^{X''}_*$ coincide.
	This implies that
	\[
		\Ker\partial_{\cX}=\Ker\partial_{\cX'}=\Ker\partial_{\cX''}.
	\]
	Using this fact, \cref{lem: exact} shows that the projections \eqref{eq: projection} induce isomorphisms
	\[
		Z^1(S^{X'}_*)\cong C(S^{X'}_*)/\Ker\partial_{\cX'}=C(S^{X''}_*)/\Ker\partial_{\cX''}
		=C(S^{X}_*)/\Ker\partial_{\cX}
		\cong Z^1(S^{X}_*).
	\]
	Our assertion for uniform forms follows from the fact that
	since $\cX$ and $\cX'$ are equivalent, for any $R\geq 0$, there exists $R'\geq0$
	sufficiently large that any ball of radius $R\geq 0$ in $\cX$
	is included in a ball with same center of radius $R'\geq 0$ in $\cX'$, and vice versa.
\end{proof}

Crystal lattices for the same group $G$ are equivalent.

\begin{lemma}\label{lem: compare}
	Let $\cX=(X,E)$ and $\cX'=(X',E')$ be crystal lattices for the action of $G$
	such that $X'=X$.  Then $\cX$ and $\cX'$ are equivalent.
\end{lemma}

\begin{proof}
	We choose a representative $\La_0$ of the orbits of $X$ with respect to the action of $G$.
	Then the set $E_{\La_0}\cup\partial E_{\La_0}$ gives a representative of the orbits of
	$E$ for the action of $G$,
	and the set $E'_{\La_0}\cup\partial E'_{\La_0}$ gives a representative of the orbits of
	$E'$ for the action of $G$.
	We let
	\begin{align*}
		C&\coloneqq\max\{ d_{\cX}(o(e'),t(e')) \mid e'\in E'_{\La_0}\cup\partial E'_{\La_0} \},\\
		C'&\coloneqq\max\{ d_{\cX'}(o(e),t(e)) \mid e\in E_{\La_0}\cup\partial E_{\La_0} \}.
	\end{align*}
	Then for any $x,y\in X$, we have
	\begin{align*}
		d_{\cX}(x,y)&\leq C d_{\cX'}(x,y),&d_{\cX'}(x,y)&\leq C' d_{\cX}(x,y),
	\end{align*}
	which proves that $\cX$ and $\cX'$ are equivalent.
\end{proof} 

We now show that any crystal lattice is equivalent to a crystal lattice which is essentially Euclidean.

\begin{lemma}\label{lem: existence}
	Let $\cX=(X,E)$ be a crystal lattice for the action of a group $G$,
	Then if we let $F\subset G$ be the free part of $G$, then
	there exists a crystal lattice $\cX'=(X',E')$ for the action of $F$
	which is essentially Euclidean.
\end{lemma}

\begin{proof}
	We let $X'=X$.
	Note that the action of $F$ on $\cX$ induces an action of $F$ on $X$ and $X\times X$.
	We we fix a fundamental domain $\La_F$ of the action of $F$ on $X$.
	We fix a free generator $\cS^{+}=\{\sigma_1,\ldots,\sigma_d\}$ of $F$,
	and we let $\cS\coloneqq\{\sigma_1^{\pm},\ldots,\sigma_d^{\pm}\}$.
	We let $E'_{\La_F}\coloneqq\{e=(x,y)\in\La_F\times\La_F\mid x\neq y\}$
	so that $(\La_F,E'_{\La_F})$ is a complete graph, and we let
	\[
		\partial E'_{\La_F}\coloneqq\{e=(x,y)\in X\times X
		\mid x\in \La_F, \exists\sigma\in\cS,\, y\in\sigma(\La_F)\}.
	\]
	Consider the disjoint union
	\[
		E'\coloneqq \bigcup_{\tau\in F} \tau\bigl(E'_{\La_F}\cup\partial E'_{\La_F}\bigr)\subset X\times X
	\]
	with natural map $o,t\colon E'\rightarrow X$ given by the first and second projections.
	Then by construction,
	\[
		\cX'=(X,E')	
	\]
	is a graph which is a crystal lattice for the action of $F$.
	Moreover, by construction, $\cX'$ satisfies \eqref{eq: EEC}
	for the coordinate $(\La_F,\cS^+)$,
	hence is essentially Euclidean. 
	By \cref{lem: compare}, $\cX$ is equivalent to $\cX'$ 
	since both $\cX$ and $\cX'$ are crystal lattices for the action of $F$.	
	This proves our assertion.
\end{proof}

We are now ready to prove \cref{thm: 1}.

\begin{proof}[Proof of \cref{thm: 1}]
	By \cref{lem: existence}, there exists a crystal lattice $\cX'=(X',E')$
	which is essentially Euclidean and equivalent to $\cX$.
	Since $\cX'$ is essentially Euclidean, by \cref{thm: 0},
	the differential $\partial_{\cX'}$ of $\cX'$ is surjective.
	Our assertion now follows from \cref{lem: equivalent}.	
\end{proof}

%
\subsection{Proof of the Main Theorem}\label{subsec: proof}
%

We are now ready prove our main theorem, \cref{thm: main}.
Let the notations be as in \S\ref{subsec: main}.
In particular, let $\cX=(X,E)$ be a general crystal lattice for the action of a finite generated abelian group $G$,
and let $(S,\phi)$ be an interaction which is irreducibly quantified.
We assume in addition that $(S,\phi)$ is simple if $\dim\cX=1$.

\begin{lemma}
	We have $\cE\oplus\partial_{\cX}\cV\subset\cC$.
\end{lemma}

\begin{proof}
	Let $\omega\in\cE\cap\partial_{\cX}\cV$.  
	Since $\omega\in\cE$, there exists $f\in C^0_\unif(S^X)^G$ such that
	$\partial_{\cX} f=\omega$.  Moreover, since $\omega\in\partial_{\cX}\cV$,
	by \cref{lem: cd},
	there exists $\zeta^{(1)},\ldots,\zeta^{(d)}\in\C{S}$
	such that $\partial_\cX f=\partial_\cX\frA$
	for $\frA=\sum_{j=1}^d\frA^j_{\zeta^{(j)}}\in\cV$.
	Since $\partial_\cX(f-\frA)=0$,
	by \cref{thm: old}, we have
	$f-\frA=\xi_X$ for some $\xi\in\C{S}$.
	We have $\xi_X\in C^0_\unif(S^X)^G$
	since $\xi_X$ is invariant with respect to the action of $G$.
	Noting that $f$ is also invariant with respect to the action of $G$, we have
	\[
		\zeta^{(j)}_X=(1-\sigma_j)\frA=(\sigma_j-1)(f-\frA)=(\sigma_j-1)\xi_X=0,
	\]
	which implies that $\zeta^{(j)}=0$ for $j=1,\ldots,d$,
	hence that $\omega=\partial_\cX\frA=0$ as desired.
\end{proof}

We now prove \cref{thm: main}.

\begin{proof}[Proof of \cref{thm: main}]
	Let $\omega\in\cC= Z^1_\unif(S^X)^G$ be a shift-invariant closed uniform form.
	By \cref{thm: 1}, there exists $f\in C^0_\unif(S^X)$ such that $\partial_{\cX}f=\omega$.
	For any $j=1,\ldots,d$, we have
	\[
		\partial_{\cX}((1-\sigma_j)f)=(1-\sigma_j)\partial_{\cX} f=(1-\sigma_j)\omega=0
	\]
	since $\omega$ is a shift-invariant form.  By \cref{thm: old}, we have
	\[
		(1-\sigma_j)f\in\Ker\partial_{\cX}\cong\C{S},
	\]
	hence $(1-\sigma_j)f=\zeta^{(j)}_X$ for some $\zeta^{(j)}\in\C{S}$.
	If we let
	\[
		g\coloneqq f-\sum_{j=1}^d\frA^j_{\zeta^{(j)}}\in C^0_\unif(S^X),
	\]
	then by \cref{lem: tau}, we have
	\[
		(1-\sigma_k)g=(1-\sigma_k)f-\sum_{j=1}^d (1-\sigma_k)\frA^j_{\zeta^{(j)}}
		=\zeta^{(k)}_X-\zeta^{(k)}_X=0
	\]
	for any $k=1,\ldots,d$.  This shows that $g\in C^0_\unif(S^X)^G$, hence we have
	\[
		\omega=\partial_{\cX} f= \partial_{\cX} g+\sum_{j=1}^d\frA^j_{\zeta^{(j)}}
		\in \cE\oplus\partial_{\cX}\cV.
	\]
	This proves our assertion.
\end{proof}

\begin{remark}
	By \cref{thm: old} and \cref{thm: 1}, for any $\cX=(X,E)$ and $(S,\phi)$ satisfying the 
	conditions of \cref{thm: 1}, we have an exact sequence
	\[
		\xymatrix{
			0\ar[r]&\C{S}\ar[r]&C^0_\unif(S^X)\ar[r]^{\partial_{\cX}}&Z^1_\unif(S^X)\ar[r]&0.
		}
	\]
	One may prove that
	the long  exact sequence associated to the group cohomology of $G$ gives an exact sequence
	\[
		\xymatrix{
		0\ar[r]&\C{S}\ar[r]&C^0_\unif(S^X)^G\ar[r]^{\partial_{\cX}}&Z^1_\unif(S^X)^G\ar[r]^-\delta& H^1(G, \C{S})\ar[r]&0.
		}
	\]
	Since $G$ acts trivially on $\C{S}$, we have 
	\[
		H^1(G, \C{S})\cong\Hom(G,\C{S})\cong\bigoplus_{j=1}^d\C{S}.
	\]
	The isomorphism	of \cref{lem: cd}
	\[
		\bigoplus_{j=1}^d\C{S}\xrightarrow\cong\partial_{\cX}\cV\subset Z^1_\unif(S^X)^G
	\]
	gives a section of  $\delta$ with image $\partial_{\cX}\cV$ giving the decomposition $\cC=\cE\oplus\partial_{\cX}\cV$
	of \cref{thm: main}.
\end{remark}



\end{document}